\documentclass[11pt]{article}

\usepackage[utf8]{inputenc}
\usepackage[normalem]{ulem}
\usepackage{amsmath}
\usepackage{amsthm}
\usepackage{amsfonts}
\usepackage{amssymb}
\usepackage{mathrsfs}
\usepackage{graphicx}
\usepackage{enumerate}
\usepackage{url}
\usepackage{hyperref}
\usepackage{color, soul}
\usepackage[usenames,dvipsnames]{xcolor}
\usepackage{authblk}
\usepackage{geometry}
\usepackage{tikz}
\usepackage{caption}
\usepackage{subcaption}
\usepackage{mathtools}
\usepackage[backend=bibtex,style=numeric,sorting=nty]{biblatex}

\usetikzlibrary{positioning,calc}
\tikzset{
    vertex/.style={
        circle,fill=black,scale=0.5
        },
    blank/.style={
        circle,fill=white
        },
    }

\usetikzlibrary{arrows}
\tikzset{
    vertex/.style={
        circle,fill=black,scale=0.5
        },
    blank/.style={
        circle,fill=white
        },
    }

\definecolor{rvwvcq}{rgb}{0,0,0}
\definecolor{Gray}{RGB}{160,160,160}

\newcommand{\R}{\mathcal{R}}
\newcommand{\N}{\mathbb{N}}
\newcommand{\A}{\mathcal{A}}
\newcommand{\wt}[1]{\widetilde{#1}}
\newcommand{\mc}[1]{\mathcal{#1}}
\newcommand{\diam}{\text{diam}}
\newcommand{\dist}{\text{dist}}
\newcommand{\SPK}{\text{SPK}}
\newcommand{\len}{\text{len}}

\DeclarePairedDelimiter{\ceil}{\lceil}{\rceil}

\newtheorem{theo}{Theorem}[section]
\newtheorem{lemma}[theo]{Lemma}
\newtheorem{prop}[theo]{Proposition}
\newtheorem{coro}[theo]{Corollary}

\theoremstyle{definition}

\newtheorem{obs}[theo]{Remark}

\title{Effect of graph operations on graph associahedra\footnote{Partially supported by Math AmSud 22-MATH-02, PIP CONICET 11220200101900CO, PIP CONICET 11220220100068CO, PICT-2020-03032, PICT 2020-00549, PICT 2020-04064, PID-UNR 80020210300068UR, PROICO 03-0723, PROIPRO 03-2923}}

\author[1,2]{Ana Gargantini\footnote{e-mail: agargantini@fcen.uncu.edu.ar}}
\author[3,4]{Adri\'an Pastine\footnote{e-mail: agpastine@gmail.com}}
\author[2,5]{Pablo Torres\footnote{e-mail: ptorres@fceia.unr.edu.ar}}

\date{ }

\affil[1]{\footnotesize Facultad de Ciencias Exactas y Naturales, Universidad Nacional de Cuyo}
\affil[2]{\footnotesize Consejo Nacional de Investigaciones Científicas y Técnicas, Argentina}
\affil[3]{\footnotesize Instituto de Matem\'atica Aplicada San Luis (UNSL-CONICET)}
\affil[4]{\footnotesize Departamento de Matem\'atica, Universidad Nacional de San Luis}
\affil[5]{\footnotesize Depto. de Matem\'atica, Facultad de Ciencias Exactas, Ing. y Agrimensura, Universidad Nacional de Rosario}

\begin{document}

\maketitle

\begin{abstract}
Given a graph $G$, we determine the structure of the rotation graph of a graph obtained by applying certain operations to $G$. Specifically, we consider the operations of adding a simplicial vertex, adding a true twin to a vertex, and the two closely related operations of deleting the set of edges from a subgraph induced by a set of true twins, and adding a false twin to a vertex. We describe how applying these operations to a graph affects the structure of its rotation graph. Furthermore, using this description, we study chromatic number, distance, and diameter in rotation graphs. In particular, we establish conditions under which the chromatic number of the rotation graphs is preserved. As an interesting consequence, we obtain that the chromatic number of the rotation graphs of threshold graphs (which includes complete split graphs and star graphs) and of complete bipartite graph is 3. We also provide a new lower bound for $\diam(\R(G-S))$ in terms of $\diam(\R(G))$, where $S$ is the set of edges of the subgraph of $G$ induced by a set of true twins. As a consequence, we improve the known lower bound for the diameter of the rotation graph of balanced complete bipartite graphs, allowing us to compute the exact value of $\diam(\mathcal{R}(K_{2,q}))$ for $q\in\{3,4,5,6,7,8\}$.
\end{abstract}

\section{Introduction}

Given a connected graph $G$, the graph associahedron $\A(G)$ of $G$ is the convex polytope whose face poset is isomorphic to the inclusion order of tubings on $G$ \cite{CDS-2006}. Several families of graph associahedra correspond to well-known polytopes. For instance, when $G$ is a complete graph, a path, a cycle or a star, the graph associahedron $\A(G)$ is the permutohedron, the classical associahedron, the cyclohedron, or the stellohedron, respectively.

Moreover, graph associahedra find relevance beyond generalization since they provide a tool for the study of combinatorial structures and their geometric representation, finding applications in various domains including optimization \cite{KLS-2010}, hierarchical visualization systems \cite{SSV-2017}, random structure generation \cite{DPR-2020}, and probabilistic methods \cite{MUWY-2018}. Additionally, they hold significance in algebra and physics as specific instances of generalized permutohedra \cite{AA-2023,PRW-2008}. One particularly compelling direction of study is the investigation on the graph $\R(G)$ determined by the $1$-skeleton of $\A(G)$, and its graph properties.  There is a number of equivalent formulations for this structure. Among them, in this article we approach the study of $\R(G)$ following \cite{CLP-2018}, considering its representation as the rotation graph of $G$, that is, as the graph whose vertex set is the set of search trees on $G$ and whose edges are determined by rotations on search trees.

Graph properties of $\R(G)$, such as hamiltonicity, connectivity and diameter have been extensively considered. For example, the diameter of permuthohedra, classical associahedra \cite{Pou-2014}, cyclohedra \cite{Pou-2017} and stellohedra \cite{MP-2015} were computed. Besides, in \cite{CPV-2022a} the authors give the exact diameter of the rotation graph of complete split graphs and of unbalanced complete bipartite graphs. In the same article, lower and upper bounds are provided for the diameter of the rotation graph of trivially perfect graphs and for the diameter of rotation graphs $\R(G)$ in terms of the order of $G$ and its tree-depth, treewidth, and pathwidth. Another challenging problem on rotation graphs is coloring. In \cite{FFHHWU-2009} the authors consider the chromatic number of the classical associahedron. They prove that it is bounded above by $\ceil {\frac{n}{2}}$ and that it is $O(n/ \log n)$ if $G$ is the path on $n-2$ vertices. Moreover, they conjecture that it is $O(\log n)$, which was proved in \cite{ARSW-2018}. The same paper mentions that the best known lower bound is $\chi(\R(P_{10}))\geq4$. Apart from these results for the classical associahedra and the fact that permutohedra are bipartite, computing the chromatic number of $\R(G)$ for the general case remains an open problem.

An interesting tool to study the structure of rotation graphs is based on the reflection on $\R(G)$ of an operation on $G$. For instance, in \cite{CDS-2011}, pseudograph associahedra are introduced, and consider deformations of pseudograph associahedra as their underlying graphs are altered by edge contractions and edge deletions. Manneville and Pilaud \cite{MP-2015} proved that if $G$ has at least two edges, its rotation graph is Hamiltonian. However, their proof does not derive in an efficient algorithm for computing a Hamiltonian cycle. Note that $|V(\R(G))|$ is in general exponential in $|V(G)|$. In 2022, Cardinal et al. \cite{CMM-2022} obtained an efficient algorithm to construct Hamiltonian paths in rotation graphs of chordal graphs by applying a combinatorial generation framework proposed in \cite{HHMW-2022}. For this graph class, the authors studied search trees of $G$ and the relationship with a perfect elimination ordering of the vertices of $G$. 

In this regard, we consider the effect of three different graph operations on $G$ on the structure of $\R(G)$. More specifically, given a subset $K\in V(G)$ that induces a clique, we consider the graph $G_K$ obtained by adding a simplicial vertex whose neighbourhood is $K$, and given $v\in V(G)$ we consider the graphs $G_v$ and $\wt G_v$ obtained from $G$ by adding a true twin of $v$, and adding a false twin of $v$, respectively. It is known that these operations generate interesting graph classes, such as  threshold graphs, cographs, Ptolemaic graphs, and distance-hereditary graphs \cite{DHgraphs}. The fact that $\wt G_v$ can be obtained from $G_v$ by deleting and edge connecting true twins, serves as a motivation to consider the operation on a graph $G$ with a set of true twins $W\subseteq V(G)$, consisting in deleting the set $S$ of edges connecting vertices from $W$. 

In Section \ref{section_operations}, we analyze the structure of the rotation graphs $\R(G_K)$, $\R(G_v)$, $\R(G-S)$, and  $\R(\wt G_v)$, characterizing the vertices and edges of these rotation graphs in terms of those in $\R(G)$. In order to do this, for $\R(G_K)$ and $\R(G_v)$, we determine all search trees on $G_K$ and $G_v$, respectively, through vertex insertions on search trees on $G$. We also establish a relation between adjacencies in $\R(G_K)$ and $\R(G_v)$ with those in $\R(G)$. Furthermore, we prove that $\R(G-S)$ is a graph quotient of $\R(G)$, which extends to the particular case of $\R(\wt G_v)$.

In Section \ref{section_chromatic_number}, we present results concerning the chromatic number of rotation graphs for graphs obtained by applying these operations to non-complete graphs. Specifically, we prove that if $K$ is a set of universal vertices in $G$, then $\chi(\R(G_K)) = \chi(\R(G))$. Additionally, if $v$ is a universal vertex in $G$, then $\chi(\R(G_v)) = \chi(\R(G))$. As a consequence, we determine that the chromatic number of $\R(G)$ for any non-complete threshold graph $G$ is 3, thereby establishing the chromatic number of stellohedra and rotation graphs of split complete graphs as particular instances of threshold graphs. We also prove that for certain graphs $G$ and $v\in V(G)$, $\chi(\R(\wt G_v)=\chi(\R(G))$, obtaining as a corollary that the chromatic number of rotation graphs of complete bipartite graphs is also 3.

In Section \ref{section_dist_and_diam}, we present results on distances between specific pairs of vertices in $G_v$. We also establish a relationship between distances in $\R(G)$ and $\R(G - S)$, where $S$ is the set of edges connecting vertices from a set of true twins. Using this relation, we provide a lower bound for $\diam(\R(G - S))$ in terms of $\diam(\R(G))$. For balanced complete bipartite graphs, this bound improves upon the one given in \cite{CPV-2022a}, although the exact value of the diameter remains unknown. We apply these results to compute $\diam(\R(K_{2,q}))$ for $q \in \{3,4,5,6,7,8\}$.

\subsection{Preliminaries}

\subsubsection*{Basic notions and notation}

In this manuscript, all graphs are simple and undirected, and have at least one vertex. Given a graph $G$, $V(G)$ and $E(G)$ are its vertex set and edge set, respectively. We denote an edge $\{u,v\}$ by $uv$ and the distance in $G$ between vertices $u$ and $v$ by $\dist_G(u,v)$.

Recall that the open neighbourhood of $v\in V(G)$ is $N_G(v)=\{u\in V(G)\mid uv\in E(G)\}$ and the closed neighbourhood of $v$ is $N_G[v]=N_G(v)\cup \{v\}$. If the graph is clear from the context, we drop the subscript $G$. Two vertices $u,v\in V(G)$ are false twins if $N_G(v)=N_G(u)$ and they are true twins if $N_G[u]=N_G[v]$. If $N_G(v)$ is a clique of $G$, then $v$ is said to be simplicial in $G$. We say that $v$ is a pendant vertex in $G$ if $|N_G(v)|=1$. Notice that a pendant vertex is simplicial. If there is a subset $W\subseteq V(G)$ such that for every $u,v\in W$, $u$ and $v$ are true twins, we say that $W$ is a set of true twins in $G$.

We use some standard terminology related to rooted trees and introduce notation. Let $T$ be a rooted tree and $v\in V(T)$. We denote by $r_T$ the root of $T$ and $d_{T,v}=\dist_T(r_T,v)$. 
We say that $v$ is in level $i$ in $T$ if $d_{T,v}=i$, and we denote $\mc L_i(T)$ the set of vertices of level $i$ in $T$. A branch in $T$ is a path from $r_T$ to a leaf of $T$. We denote $T|v$ the subtree of $T$ rooted in $v$. If $u$ is the parent of $v$ in $T$, we say that $T|v$ is a subtree of $u$.
If $u, v$ are two vertices such that $u$ belongs to the path in $T$ from the root to $v$, then $u$ is an ancestor of $v$ and $v$ is a descendant of $u$. 

Given two positive integers $p$ and $q$, the complete split graph $\SPK_{p,q}$ is the graph whose vertex set is $P\cup Q$ where $P=\{x_1,\ldots,x_p\}, Q=\{y_1,\ldots,y_q\}$ are disjoint sets, such that $P$ is a clique and $Q$ is an independent set and where $\{x_i,y_j\mid 1\leq i \leq p, 1\leq j\leq q\}\subseteq E(\SPK{p,q})$.


\subsubsection*{Search trees, rotations and rotation graphs}

A search tree $T$ on a connected graph $G$ is a rooted tree with vertex set $V(G)$ defined recursively as follows. The root of $T$ is a vertex $r\in V(G)$ and the children of $r$ are the roots of search trees on each connected component of $G-r$ \cite{CPV-2022a}.

Notice that a search tree $T$ on $G$ can be obtained by succesively deleting the vertices of $G$ following an elimination order given by a permutation of $V(G)$. For this reason the search trees on $G$ are also known as elimination trees. Observe that different eliminations orders may result in the same elimination tree. Figure \ref{fig:exampleSpecial} shows three search trees on the split complete graph $\SPK_{3,3}$. Note that $R$ and $S$ are determined by unique elimination orders $y_1,y_2,y_3,x_1,x_2,x_3$ and $y_1,y_2,y_3,x_1,x_3,x_2$, respectively. On the other hand, $T$ is determined by two different elimination orders, $x_1,y_1,x_2,x_3,y_2,y_3$ and $x_1,y_1,x_2,x_3,y_3,y_2$.

\begin{figure}[ht]
    \centering
    \begin{tikzpicture}
\foreach \i/\u in {1/below,2/left,3/above}{
\node[vertex,label=\u:{\small $x_{\i}$}] at (0,\i) (x\i) {};
}
\foreach \i/\u in {1/below,2/right,3/above}{
\node[vertex,label=\u:{\small $y_{\i}$}] at (1.3,\i) (y\i) {};
}
\foreach \i in {1,2,3}{
\draw (x1)--(y\i);
\draw (x2)--(y\i);
\draw (x3)--(y\i);
}
\draw (x1)--(x2)--(x3);
\draw[-] (x1) to [bend left=90] (x3); 
\node at (0.6,0) (G) {$\SPK_{3,3}$};
\end{tikzpicture}
\qquad
\begin{tikzpicture}
\node[vertex,label=right:{\small $y_1$}] at (0,3.5) (y1) {};
\node[vertex,label=right:{\small $y_2$}] at (0,3) (y2) {};
\node[vertex,label=right:{\small $y_3$}] at (0,2.5) (y3) {};
\node[vertex,label=right:{\small $x_1$}] at (0,2) (x1) {};
\node[vertex,label=right:{\small $x_2$}] at (0,1.5) (x2) {};
\node[vertex,label=right:{\small $x_3$}] at (0,1) (x3) {};
\draw (y1)--(y2)--(y3)--(x1)--(x2)--(x3);
\node at (0,0.3) (R) {$R$};
\end{tikzpicture}
\quad
\begin{tikzpicture}
\node[vertex,label=right:{\small $y_1$}] at (0,3.5) (y1) {};
\node[vertex,label=right:{\small $y_2$}] at (0,3) (y2) {};
\node[vertex,label=right:{\small $y_3$}] at (0,2.5) (y3) {};
\node[vertex,label=right:{\small $x_1$}] at (0,2) (x1) {};
\node[vertex,label=right:{\small $x_3$}] at (0,1.5) (x3) {};
\node[vertex,label=right:{\small $x_2$}] at (0,1) (x2) {};
\draw (y1)--(y2)--(y3)--(x1)--(x3)--(x2);
\node at (0,0.3) (S) {$S$};
\end{tikzpicture}
\quad
\begin{tikzpicture}
\node[vertex,label=right:{\small $x_1$}] at (0,3.5) (x1) {};
\node[vertex,label=right:{\small $y_1$}] at (0,3) (y1) {};
\node[vertex,label=right:{\small $x_2$}] at (0,2.5) (y2) {};
\node[vertex,label=right:{\small $x_3$}] at (0,2) (y3) {};
\node[vertex,label=right:{\small $y_2$}] at (-0.6,1) (x2) {};
\node[vertex,label=right:{\small $y_3$}] at (0.6,1) (x3) {};
\draw (x1)--(y1)--(y2)--(y3)--(x2);
\draw (y3)--(x3);
\node at (0,0.3) (T) {$T$};
\end{tikzpicture}
    \caption{Three search trees on $\SPK_{3,3}$.}
    \label{fig:exampleSpecial}
\end{figure}


In the subsequent discussions, we make use of the following remarks regarding search trees on graphs.

\begin{obs}
\label{obs_search_subtree}
Let $G$ be a connected graph and $T$ be a rooted tree with vertices $V(G)$. The following are equivalent.
\begin{enumerate}[(a)]
    \item $T$ is a search tree on $G$.
    \item For every $v\in V(T)$ such that $v\neq r_T$, $V(T|v)$ induces a connected component in $G-\bigcup\limits_{j=0}^{d_{T,v}-1}\mathcal L_j(T)$.
\end{enumerate}
\end{obs}

\begin{obs}
\label{obs_uv_branch}    
    A pair of vertices $u,v$ are adjacent in $G$ if and only if for every search tree $T$ on $G$, $u$ and $v$ are in the same branch in $T$. 


\end{obs}

Following \cite{CLP-2018}, we recall the definition of rotation in a search tree. Let $G$ be a connected graph and $T$ a search tree on $G$. Let $u, v\in V(T)$ such that $v$ is a child of $u$ in $T$ and $p$ be the parent of $u$ in $T$. The rotation of $u$ and $v$ (or $uv$-rotation) in $T$ produces a search tree $T'$ on $G$ where
\begin{itemize}
    \item $u$ is a child of $v$ and $v$ is a child of $p$,
    \item every subtree $S$ of $u$ in $T$ such that $S\neq T|v$, is a subtree of $u$ in $T'$,
    \item for every subtree $S$ of $v$ in $T$, if $u$ is adjacent to a vertex of $S$ in $G$, then $S$ is a subtree of $u$ in $T'$; otherwise, $S$ is a subtree of $v$ in $T'$.
\end{itemize}

The rotation graph $\R(G)$ of $G$ is the graph whose vertex set is the set of all the search trees on $G$ and where two search trees are adjacent if and only if they differ by one rotation.  By definition of $\R(G)$, when referring to a search tree $T$ on $G$, we often write $T\in V(\R(G))$. A sequence of rotations that transform a search tree $T$ on $G$ into another $T'$ determines a $TT'$-walk in $\R(G)$, that is, a sequence of vertices $T=T_0,T_1,\ldots,T_k=T'$ in $\R(G)$ such that two consecutive are adjacent. Similarly, a $TT'$-path in $\R(G)$ is a $TT'$-walk with distinct vertices. The rotation distance between search trees $T,T'$ on $G$ is the minimum number of rotations required to transform $T$ into $T'$, that is, the distance $\dist_{\R(G)}(T,T')$ in $\R(G)$ between $T$ and $T'$.

Figure \ref{fig:GraphAssociahedra} shows the rotation graph of the complete graph $K_4$ ($1$-skeleton of the permutohedron on $4$ vertices). In this figure we denote $a_1a_2a_3a_4$ to the search tree that is a path with consecutive vertices $a_1,a_2,a_3,a_4$ rooted at $a_1$.

\begin{figure}
\centering
    \begin{subfigure}{0.2\textwidth}
    \centering
        \begin{tikzpicture}
        \node[vertex,label=left:{\small $1$}] at (0,1) (1) {};
        \node[vertex,label=right:{\small $3$}] at (1,1) (3) {};
        \node[vertex,label=left:{\small $2$}] at (0,0) (2) {};
        \node[vertex,label=right:{\small $4$}] at (1,0) (4) {};
        \draw (1)--(3)--(4)--(1)--(2)--(3);
        \draw (2)--(4);
        \end{tikzpicture}
        \caption{$K_4$}
    \end{subfigure}
\begin{subfigure}{0.4\textwidth}
\centering
\begin{tikzpicture}[line cap=round,line join=round,>=triangle 45,x=0.6cm,y=0.6cm]
\draw [line width=1pt,color=Gray] (0.25,-3.14)-- (-2.19,-3.62);
\draw [line width=1pt,color=Gray] (2.43,0.58)-- (3.65,0.3);
\draw [line width=1pt,color=Gray] (0.81,2.22)-- (0.21,3.72);
\draw [line width=1pt,color=Gray] (-3.17,0.3)-- (-0.79,0.58);
\draw [line width=1pt,color=Gray] (-3.17,0.3)-- (-3.99,-1.82);
\draw [line width=1pt] (-4.01,1.76)-- (-2.15,3.56);
\draw [line width=1pt] (-2.15,3.56)-- (-1.11,3.38);
\draw [line width=1pt] (-1.11,3.38)-- (1.41,3.56);
\draw [line width=1pt] (1.41,3.56)-- (0.21,3.72);
\draw [line width=1pt] (0.21,3.72)-- (-2.15,3.56);
\draw [line width=1pt] (-1.11,3.38)-- (-1.85,1.2);
\draw [line width=1pt] (-1.85,1.2)-- (-3.87,-0.8);
\draw [line width=1pt] (-3.87,-0.8)-- (-1.87,-2.8);
\draw [line width=1pt] (-1.87,-2.8)-- (0.15,-0.8);
\draw [line width=1pt] (0.15,-0.8)-- (-1.85,1.2);
\draw [line width=1pt,color=Gray] (0.81,2.22)-- (-0.79,0.58);
\draw [line width=1pt,color=Gray] (-0.79,0.58)-- (0.81,-1.02);
\draw [line width=1pt,color=Gray] (0.81,-1.02)-- (2.43,0.58);
\draw [line width=1pt,color=Gray] (2.43,0.58)-- (0.81,2.22);
\draw [line width=1pt] (3.21,1.74)-- (2.69,-0.38);
\draw [line width=1pt] (3.21,1.74)-- (3.65,0.3);
\draw [line width=1pt] (3.65,0.3)-- (3.21,-1.78);
\draw [line width=1pt] (3.21,-1.78)-- (2.69,-0.38);
\draw [line width=1pt,color=Gray] (0.81,-1.02)-- (0.25,-3.14);
\draw [line width=1pt,color=Gray] (0.25,-3.14)-- (1.39,-3.62);
\draw [line width=1pt] (1.39,-3.62)-- (3.21,-1.78);
\draw [line width=1pt] (0.15,-0.8)-- (2.69,-0.38);
\draw [line width=1pt] (-1.87,-2.8)-- (-1.13,-4.14);
\draw [line width=1pt] (-1.13,-4.14)-- (1.39,-3.62);
\draw [line width=1pt] (-2.19,-3.62)-- (-1.13,-4.14);
\draw [line width=1pt] (-2.19,-3.62)-- (-3.99,-1.82);
\draw [line width=1.1pt] (-3.99,-1.82)-- (-4.89,-0.36);
\draw [line width=1pt] (-4.89,-0.36)-- (-3.87,-0.8);
\draw [line width=1pt] (-4.89,-0.36)-- (-4.01,1.76);
\draw [line width=1.1pt] (-3.17,0.3)-- (-4.01,1.76);
\draw [line width=1pt] (1.41,3.56)-- (3.21,1.74);
\begin{scriptsize}
\draw [fill=rvwvcq] (-1.11,3.38) circle (2.5pt);
\draw[color=rvwvcq] (-0.98,3.9) node {3214};
\draw [fill=rvwvcq] (-2.15,3.56) circle (2.5pt);
\draw[color=rvwvcq] (-2.7,3.99) node{2314};
\draw [fill=rvwvcq] (0.21,3.72) circle (2.5pt);
\draw[color=rvwvcq] (0.37,4.15) node{2341};
\draw [fill=rvwvcq] (1.41,3.56) circle (2.5pt);
\draw[color=rvwvcq] (2,3.9) node{3241};
\draw [fill=rvwvcq] (3.21,1.74) circle (2.5pt);
\draw[color=rvwvcq] (3.9,2) node{3421};
\draw [fill=rvwvcq] (2.69,-0.38) circle (2.5pt);
\draw (3.5,-0.38) node[rectangle,minimum height=0.3cm,fill=white, fill opacity=0.8] { };
\draw[color=rvwvcq] (3.38,-0.38) node {3412};
\draw [fill=rvwvcq] (3.21,-1.78) circle (2.5pt);
\draw[color=rvwvcq] (3.89,-1.88) node{4312};
\draw [fill=rvwvcq] (3.65,0.3) circle (2.5pt);
\draw[color=rvwvcq] (4.4,0.3) node{4321};
\draw [fill=rvwvcq] (2.43,0.58) circle (2.5pt);
\draw[color=rvwvcq] (1.7,0.58) node{4231};
\draw [fill=rvwvcq] (0.81,2.22) circle (2.5pt);
\draw[color=rvwvcq] (0.1,2.22) node{2431};
\draw [fill=rvwvcq] (-0.79,0.58) circle (2.5pt);
\draw[color=rvwvcq] (-0.05,0.58) node{2413};
\draw [fill=rvwvcq] (0.81,-1.02) circle (2.5pt);
\draw[color=rvwvcq] (1.4,-1.2) node{4213};
\draw [fill=rvwvcq] (0.25,-3.14) circle (2.5pt);
\draw[color=rvwvcq] (0.95,-3) node{4123};
\draw [fill=rvwvcq] (1.39,-3.62) circle (2.5pt);
\draw[color=rvwvcq] (1.55,-4.05) node{4132};
\draw [fill=rvwvcq] (-1.13,-4.14) circle (2.5pt);
\draw[color=rvwvcq] (-0.97,-4.58) node{1432};
\draw [fill=rvwvcq] (-2.19,-3.62) circle (2.5pt);
\draw[color=rvwvcq] (-2.3,-4.05) node{1423};
\draw [fill=rvwvcq] (-1.87,-2.8) circle (2.5pt);
\draw[color=rvwvcq] (-1.18,-2.83) node {1342};
\draw [fill=rvwvcq] (0.15,-0.8) circle (2.5pt);
\draw[color=rvwvcq] (-0.63,-0.8) node {3142};
\draw [fill=rvwvcq] (-1.85,1.2) circle (2.5pt);
\draw[color=rvwvcq] (-2.4,1.6) node {3124};
\draw [fill=rvwvcq] (-3.87,-0.8) circle (2.5pt);
\draw[color=rvwvcq] (-3.11,-0.8) node {1324};
\draw [fill=rvwvcq] (-4.89,-0.36) circle (2.5pt);
\draw[color=rvwvcq] (-5.65,-0.3) node {1234};
\draw [fill=rvwvcq] (-4.01,1.76) circle (2.5pt);
\draw[color=rvwvcq] (-4.55,2.1) node {2134};
\draw [fill=rvwvcq] (-3.17,0.3) circle (2.5pt);
\draw[color=rvwvcq] (-3.86,0.23) node {2143};
\draw [fill=rvwvcq] (-3.99,-1.82) circle (2.5pt);
\draw[color=rvwvcq] (-4.35,-2.25) node {1243};
\end{scriptsize}
\end{tikzpicture}
\caption{$\R(K_4)$}
\end{subfigure}
\caption{A graph and its corresponding rotation graph.}
\label{fig:GraphAssociahedra}
\end{figure}

\subsubsection*{Insertion and elimination of a vertex from a rooted tree}

We consider the insertion and elimination operations as defined in \cite{CMM-2021} for search trees on chordal graphs. Here, we extend the discussion of these operations and their associated results to a more general setting.

Let $T$ be a rooted tree with more than one vertex, $v\in V(T)$, and  $i\in\{0,\ldots,d_{T,v}\}$. Assuming $x$ is not a vertex of $T$, we describe the operation of insertion of the vertex $x$ along the path between $r_T$ and $v$. Formally, we denote by $T(i,x,v)$ the rooted tree with vertex set $V(T)\cup\{x\}$ defined as follows. 
\begin{itemize}
    \item $T(0,x,v)$ is the rooted tree with root $x$ and such that $T$ is the only subtree of $x$.
    \item $T(d_{T,v}+1, x,v)$ the rooted tree with root $r_T$ obtained from $T$ by adding $x$ as a new child of $v$.
    \item If $d_{T,v}\geq 1$, let $A$ be the path in $T$ from $r_T$ to $v$ and suppose $A$ has vertices $r_T=a_0,\ldots,a_{d_{T,v}}=v$. For $i\in\{1,\ldots,d_{T,v}\}$, $T(i,x,v)$ is the rooted tree with root $r_T$ obtained from $T$ by subdividing the edge $a_{i-1}a_i$ and labelling with $x$ the vertex added in the subdivision.
\end{itemize}


\begin{figure}
    \centering
    \begin{tikzpicture}
        \node[vertex,label=right:{\small $a=a_0$}] at (0,2) (u) {};
        \node[vertex,label=right:{\small $b=a_1$}] at (0,1) (x) {};
        \node[vertex,label=right:{\small $d=a_2$}] at (0.5,0) (v) {};
        \node[vertex,label=left:{\small $c$}] at (-0.5,0) (w) {};
        \node at (0,-0.5) (T) {$T$};
        \draw (u)--(x)--(v);
        \draw (x)--(w);
    \end{tikzpicture}
    \qquad
    \begin{tikzpicture}
        \node[vertex,label=right:{\small $x$}] at (0,3) (a) {};
        \node[vertex,label=right:{\small $a$}] at (0,2) (u) {};
        \node[vertex,label=right:{\small $b$}] at (0,1) (x) {};
        \node[vertex,label=right:{\small $d$}] at (0.5,0) (v) {};
        \node[vertex,label=left:{\small $c$}] at (-0.5,0) (w) {};
        \node at (0,-0.5) (T0a) {$T(0,x,d)$};
        \draw (a)--(u)--(x)--(v);
        \draw (x)--(w);
    \end{tikzpicture}
    \qquad
        \begin{tikzpicture}
        \node[vertex,label=right:{\small $a$}] at (0,3) (a) {};
        \node[vertex,label=right:{\small $x$}] at (0,2) (u) {};
        \node[vertex,label=right:{\small $b$}] at (0,1) (x) {};
        \node[vertex,label=right:{\small $d$}] at (0.5,0) (v) {};
        \node[vertex,label=left:{\small $c$}] at (-0.5,0) (w) {};
        \node at (0,-0.5) (T0a) {$T(1,x,d)$};
        \draw (a)--(u)--(x)--(v);
        \draw (x)--(w);
    \end{tikzpicture}
    \qquad
    \begin{tikzpicture}
        \node[vertex,label=right:{\small $a$}] at (0,2) (u) {};
        \node[vertex,label=right:{\small $b$}] at (0,1) (x) {};
        \node[vertex,label=right:{\small $x$}] at (0.5,0) (a) {};
        \node[vertex,label=left:{\small $c$}] at (-0.5,0) (w) {};
        \node[vertex,label=right:{\small $d$}] at (0.5,-1) (v) {};
        \node at (0,-1.5) (T) {$T(2,x,d)$};
        \draw (u)--(x)--(a)--(v);
        \draw (x)--(w);
    \end{tikzpicture}
    \qquad
    \begin{tikzpicture}
        \node[vertex,label=right:{\small $a$}] at (0,2) (u) {};
        \node[vertex,label=right:{\small $b$}] at (0,1) (x) {};
        \node[vertex,label=right:{\small $d$}] at (0.5,0) (a) {};
        \node[vertex,label=left:{\small $c$}] at (-0.5,0) (w) {};
        \node[vertex,label=right:{\small $x$}] at (0.5,-1) (v) {};
        \node at (0,-1.5) (T) {$T(3,x,d)$};
        \draw (u)--(x)--(a)--(v);
        \draw (x)--(w);
    \end{tikzpicture}
    \caption{Insertions of a vertex in the rooted tree $T$}
    \label{fig:insertions}
    \end{figure}

Figure \ref{fig:insertions} shows a rooted tree $T$ and $T(i,x,d)$ for $i\in\{0,1,2,4\}$.

We also consider the operation of elimination of a vertex from $T$. If $u\in V(T)$ has at most one child in $T$, we define the rooted tree $p(T,u)$ as follows.
\begin{itemize}
    \item If $u$ is the root of $T$ and its child is $c$, $p(T,u)$ is obtained from $T$ by deleting $u$ and making $c$ the root. 
    \item If $u$ is not the root of $T$, $u$ has exactly one child $c$, and $b$ is the parent of $u$ in $T$, then $p(T,u)$ is obtained by deleting $u$ from $T$, adding the edge $bc$ and keeping $r_T$ as the root. 
    \item If $u$ is a leaf of $T$ then $p(T,u)$ is obtained from $T$ by deleting $u$ and keeping $r_T$ as the root. 
\end{itemize}

    \begin{figure}
    \centering
    \begin{tikzpicture}
        \node[vertex,label=right:{\small $1$}] at (0,3) (a) {};
        \node[vertex,label=right:{\small $2$}] at (0,2) (u) {};
        \node[vertex,label=right:{\small $3$}] at (0,1) (x) {};
        \node[vertex,label=right:{\small $5$}] at (0.5,0) (v) {};
        \node[vertex,label=left:{\small $4$}] at (-0.5,0) (w) {};
        \node at (0,-0.5) (T0a) {$T'$};
        \draw (a)--(u)--(x)--(v);
        \draw (x)--(w);
    \end{tikzpicture}
    \qquad
    \begin{tikzpicture}
        \node[vertex,label=right:{\small $2$}] at (0,2) (u) {};
        \node[vertex,label=right:{\small $3$}] at (0,1) (x) {};
        \node[vertex,label=right:{\small $5$}] at (0.5,0) (v) {};
        \node[vertex,label=left:{\small $4$}] at (-0.5,0) (w) {};
        \node at (0,-0.5) (T0a) {$p(T',1)$};
        \draw (u)--(x)--(v);
        \draw (x)--(w);
    \end{tikzpicture}
    \qquad
        \begin{tikzpicture}
        \node[vertex,label=right:{\small $1$}] at (0,3) (a) {};
        \node[vertex,label=right:{\small $3$}] at (0,1) (x) {};
        \node[vertex,label=right:{\small $5$}] at (0.5,0) (v) {};
        \node[vertex,label=left:{\small $4$}] at (-0.5,0) (w) {};
        \node at (0,-0.5) (T0a) {$p(T',2)$};
        \draw (a)--(x)--(v);
        \draw (x)--(w);
    \end{tikzpicture}
    \qquad
    \begin{tikzpicture}
        \node[vertex,label=right:{\small $1$}] at (0,3) (a) {};
        \node[vertex,label=right:{\small $2$}] at (0,2) (u) {};
        \node[vertex,label=right:{\small $3$}] at (0,1) (x) {};
        \node[vertex,label=right:{\small $5$}] at (0.5,0) (v) {};
        \node at (0,-0.5) (T0a) {$p(T',4)$};
        \draw (a)--(u)--(x)--(v);
    \end{tikzpicture}    
    \qquad    
    \caption{Eliminations of a vertex from the rooted tree $T'$.}
    \label{fig:eliminations}
    \end{figure}

Figure \ref{fig:eliminations} a rooted tree $T'$ and $p(T',u)$ for $u\in\{1,2,4\}$.

\begin{obs}
\label{obs_ins_and_prun_are_inverses}
Notice that insertion and elimination of vertices are mutually inverse operations. More precisely, if $T$ is a rooted tree, $v\in V(T)$, and $i\in\{0,\dots,d_{T,v}+1\}$ then $p(T(i,x,v),x)=T$. On the other hand, if $T'$ is a rooted tree and $x\in V(T')$ has at most one child in $T'$, then $p(T',x)(d_{T',x},x,v)=T'$ if $v$ is a child of $x$, and $p(T',x)(d_{T',x}+1,x,p)=T'$ if $x$ has no children and $p$ is the parent of $x$ in $T'$.
\end{obs}

\section{Operations in graphs and rotation graphs}
\label{section_operations}

In this section, we consider operations on graphs and analyze their effect on the structure of the corresponding rotation graphs. Specifically, for a graph $G'$ obtained from a graph $G$ by applying one of these operations, we aim to study vertices and edges in the rotation graph $\mathcal{R}(G')$ in terms of those in $\mathcal{R}(G)$. We begin by determining conditions under which the rooted trees obtained by insertion and elimination of vertices also function as search trees. The following two lemmas are analogous to Lemmas 19 and 20 presented in \cite{CMM-2021}, but stated for a more general context.

Recall that a $v$ is a cut vertex of a graph $G$ if $G-v$ has more connected components than $G$.

\begin{lemma}
\label{obs_pTx_is_ST}
Let $G'$ be a connected graph and $x,v\in V(G')$ such that $N_{G'}[x]\subseteq N_{G'}[v]$. If $T'$ is a search tree on $G'$ such that $v$ is a descendant of $x$, then $p(T',x)$ is a search tree on $G'-x$.
\end{lemma}

\begin{proof}
    
    
    First, observe that for every $k\in\{0,\ldots,d_{T',x}-1\}$, $x,v\in G'-\bigcup\limits_{j=0}^k \mc L_j(T')$, and since $N_{G'}[x]\subseteq N_{G'}[v]$, $x$ is not a cut vertex of that subgraph. 
    
    This implies, on the one hand, that  $x$ has at most one child in $T'$ and therefore $p(T',x)$ is well defined. On the other hand, let $T=p(T',x)$ and let us see that $T$ is a search tree on $G'-x$. If $u$ is an ancestor of $x$ in $T'$, then $V(T|u)=V(T'|u)-\{x\}$. In addition, if $u$ is not an ancestor of $x$, then $V(T|u)=V(T'|u)=V(T'|u)-\{x\}$. In both cases, by the previous observation, $V(T|u)$ induces a connected component of $(G'-x)-\bigcup\limits_{j=0}^{d_{T,u}-1}\mc L_j(T')$. Therefore, by Remark \ref{obs_search_subtree}, $T$ is a search tree on $G'-x$.
\end{proof}

\begin{lemma}
\label{obs_Tix_are_ST_and_path}
Let $G'$ be a connected graph and $x,v\in V(G')$ such that $N_{G'}[x]\subseteq N_{G'}[v]$. Let $T$ be a search tree on $G=G'-x$ and $i\in\{0,\ldots,d_{T,v}\}$. Then $T(i,x,v)$ is a search tree on $G'$. In addition, the vertices $T(i,x,v)$ with $i\in\{0,\ldots,d_{T,v}\}$ are on a path in $\R(G')$.
\end{lemma}

\begin{proof}
    First, notice that $x$ is not a cut vertex of $G'$. Let $T'(i)=T(i,x,v)$ for $i\in\{0,\ldots,d_{T,v}\}$. For every $u\in V(T'(i))$, the vertex set of $T'(i)|u$ is $V(T|u)\cup\{x\}$ if $u$ is an ancestor of $x$ in $T'(i)$ or $u=x$, and $V(T|u)$ otherwise.

    In any case, $T'(i)|u$ induces a connected component of $G'-\bigcup\limits_{j=0}^{d_{T'(i),u}-1} \mc L_j(T'(i))$ since $x$ is not a cut vertex of $T'(i)|u$. Thus, $T'(i)$ is a search tree on $G'$.

    Furthermore, notice that $T'(i)$ differs from $T'(i+1)$ by a $xa_i$-rotation. That is, the vertices $T'(i)$ with $i\in\{0,\ldots,d_{T,v}\}$ are on a path in $\R(G')$.
\end{proof}

\subsection{Adding a simplicial vertex}
\label{subsection_adding_simplicial}

Let $G$ be a connected graph and $K\subseteq V(G)$ a subset of vertices that induces a clique in $G$. We denote by $G_K$ the graph determined by
\begin{align*}
    V(G_K) &= V(G)\cup\{x\},\\
    E(G_K) &= E(G)\cup\{vx \mid v\in K \},
\end{align*}
assuming $x$ is not a vertex of $G$. That is, $G_K$ is the graph obtained from $G$ by adding a simplicial vertex $x$ whose open neighbourhood is $K$. Note that $N_{G_K}[x]\subseteq N_{G_K}[v]$, for all $v\in K$.

Let $T$ be a search tree on $G$ and $\lambda(T) = \max \{d_{T,u} \mid u\in K\}$. We denote by $v_{\lambda(T)}$ the unique vertex from $K$ of level $\lambda(T)$ in $T$. Notice that this is well defined since by Remark \ref{obs_uv_branch}, all the vertices of $K$ are on the same branch on $T$. For simplicity, for every $i\in\{0,\dots,\lambda(T)\}$, in this section we denote $T(i,x,v_{\lambda(T)})$ as $T(i)$. We define
\[
P^x(T)=\{T(i)\mid T\in V(\R(G)), i\in\{0,\ldots,\lambda(T)+1\}\}.
\]

A chordal graph is known to have a perfect elimination ordering, meaning there exists a sequence of its vertices such that every vertex $v$ is a simplicial vertex in the subgraph induced by $v$ and the preceding vertices in the sequence. In \cite{CMM-2021}, the authors present an algorithm that determines $V(\R(G'))$ for a chordal graph $G'$. In the following proposition, we extend this idea to characterize the vertices of $V(\R(G_K))$, where $G_K$ may not be chordal but includes at least one simplicial vertex.

\begin{obs}
\label{obs_simplicial_only_child}
    It is not hard to see that a vertex $w\in V(G)$ has at least two nonadjacent neighbors $u,v$ in $G$ if and only if there exists a search tree of $G$ such that $w$ has at least two children. Thus, a vertex $w$ has at most one child in every search tree of $G$ if and only if $N(w)$ is a complete graph, i.e. if $w$ is a simplicial vertex in $G$.
\end{obs}


\begin{prop}
\label{prop_part_PpT}
Let $G$ be a connected graph and $K\subseteq V(G)$ a subset of vertices that induces a clique in $G$. Then $\{P^x(T) \mid T\in V(\R(G))\}$ is a partition of $V(\R(G_K))$. 
\end{prop}

\begin{proof}    
    Let $T\in V(\R(G))$, that is, $T$ is a search tree on $G_K-x$. Recall that $N_{G_K}[x]\subseteq N_{G_K}[v_{\lambda(T)}]$, and therefore by Lemma \ref{obs_Tix_are_ST_and_path}, for every $i\in\{0,\ldots, \lambda(T)\}$,  $T(i)$, is a search tree on $G_K$. In addition, note that $\{x\}$ is a connected component of $G_K-\bigcup\limits_{j=0}^{\lambda(T)}\mc L_j(T(i))$. Then $T(\lambda(T)+1)$ is also a search tree on $G_K$. Thus, for every $T\in V(\R(G))$, $P^x(T)\subseteq V(\R(G_K))$ and, therefore, $\bigcup\limits_{T\in V(\R(G))}P^x(T) \subseteq V(\R(G_K))$.
    
    Conversely, if $T'$ is a search tree on $G_K$, by Remark \ref{obs_simplicial_only_child}, $T=p(T',x)$ is well-defined. Recall that the vertices in $K$ altogether with $x$ are on a branch in $T'$. If  $v_{\lambda(T)}$ is a descendant of $x$ in $T'$, then $T$ is a search tree on $G$ by Lemma \ref{obs_pTx_is_ST}. Otherwise, since $N_{G_K}(x)=K$, $x$ is a leaf adjacent to $v_{\lambda(T)}$ and $p(T',x)$ is a search tree on $G$. Moreover, $T'=T(i)$ with $i=d_{T',x}$ by Remark \ref{obs_ins_and_prun_are_inverses}. Thus, $T'\in P^x(T)$ and therefore $V(\R(G_K))\subseteq \bigcup\limits_{T\in V(\R(G))}P^x(T)$.
    
    It remains to prove that if for some $T_1,T_2\in V(\R(G))$, $T'\in P^x(T_1)\cap P^x(T_2)$, then $T_1=T_2$. But this is true since $T_1(d_{T',x})=T'=T_2(d_{T',x})$ and therefore $T_1=p(T_1(d_{T',x}),x)=p(T_2(d_{T',x}),x)=T_2$.
\end{proof}

Now, we describe the edges in $\R(G_K)$. 

First, notice that for every $T\in V(\R(G))$ and for every $i\in \{0,\ldots,d_{T,v}-1\}$, $T(i)$ and $T(i+1)$ are adjacent in $\R(G_K)$ by Lemma $\ref{obs_Tix_are_ST_and_path}$. Additionally, $T(\lambda(T))$ and $T(\lambda(T)+1)$ are adjacent since they differ by a $v_{\lambda(T)}x$-rotation. Thus, the search trees on $P^x(T)$ are on a path in $\R(G_K)$.

Next, we study adjacencies in $\R(G_K)$ with one endpoint in $P^x(T)$ and the other in $P^x(T')$ for a pair $T,T'$ of adjacent vertices of $\R(G)$. To this effect, we analyse in the following enumeration the possible cases of rotations that determine adjacency in $\R(G)$.

Let $T\in V(\R(G))$, $k=\lambda(T)$ and $A$ the path in $T$ between $r_T$ and $v_{\lambda(T)}$. Let $T'\in V(\R(G))$ obtained from $T$ by an $ab$-rotation with $a,b$ adjacent vertices of $T$. Suppose $a$ is the parent of $b$ in $T$.

\begin{enumerate}
\item Assume $a$ and $b$ are not in $A$. Then, $r_{T'}=r_T$ and the path in $T'$ between $r_{T'}$ and $v_{\lambda(T)}$ is exactly $A$. Thus, $|P^x(T)|=|P^x(T')|$ and for every $i\in \{0,\ldots,\lambda(T)+1\}$, $T(i)$ is adjacent to $T'(i)$ since they differ by the rotation of the pair $a,b$. Moreover, note that $T(i)$ and $T'(i')$ are adjacent if and only if $i=i'$.
\begin{center}
\begin{tikzpicture}[x=1.5cm]
\foreach \x/\i in {0/0,1/1,3/k/1}{
    \node[vertex,label=above:{$T(\i)$}] at (\x,1) (T1\x) {};
    \node[vertex,label=below:{$T'(\i)$}] at (\x,0) (T2\x) {};
    \draw (\x,1) -- (\x+1,1);
    \draw (\x,0) -- (\x+1,0);
    \draw (\x,0) -- (\x,1);
    }
\draw (2,1) -- (3,1);
\draw (2,0) -- (3,0);
\node[blank] at (2,1) (T12) {$\cdots$};
\node[blank] at (2,0) (T22) {$\cdots$};
\node[vertex,label=above:{$T(k+1)$}] at (4,1) (T1k) {};
\node[vertex,label=below:{$T'(k+1)$}] at (4,0) (T2k) {};
\draw (T1k)--(T2k);
\end{tikzpicture}
\end{center}

\item Now consider the case that $a$ and $b$ are both in $A$. Then $a\neq v_{\lambda(T)}$, since $a$ is the parent of $b$ in $T$.

\begin{enumerate}
\item Suppose $b\neq v_{\lambda(T)}$. Let $S$ be the subtree of $b$ in $T$ that contains $v_{\lambda(T)}$. Suppose $A$ has vertices $r_T=a_0,a_1,\ldots,a_k=v_{\lambda(T)}$ and let $l\in\{1,\ldots,k-1\}$ be such that  $a=a_{l-1}$ and $b=a_l$. 
\begin{enumerate}
\item If there is a vertex in $S$ that is adjacent to $a$ in $G$, then $S$ is a subtree of $a$ in $T'$. Thus, $a$ remains in the path from $r_{T'}$ to $v_{\lambda(T)}$ in $T'$. Therefore $\lambda(T)=\lambda(T')$ and for every $i\in\{0,\ldots,k+1\}$ such that $i\neq l$, $T(i)$ is adjacent to $T'(i)$ since they differ by the $ab$-rotation. Also, $T(l)$ and $T'(l)$ are nonadjacent (notice that neither $T(l)$ nor $T'(l)$ contain the edge $ab$).

\begin{center}
\begin{tikzpicture}[x=1.5cm]
\foreach \x/\i/\c in {0/0/black,2/l-1/black,3/l/white,4/l+1/black}{
    \draw (\x,1) -- (\x+1,1);
    \draw (\x,0) -- (\x+1,0);
    \draw[\c] (\x,0) -- (\x,1);
    \node[vertex,label=above:{\small $T(\i)$}] at (\x,1) (T1\x) {};
    \node[vertex,label=below:{\small $T'(\i)$}] at (\x,0) (T2\x) {};
    }
\foreach \x in {1,5}{
    \draw (\x,1) -- (\x+1,1);
    \draw (\x,0) -- (\x+1,0);
    \node[blank] at (\x,1) (T1\x) {$\cdots$};
    \node[blank] at (\x,0) (T2\x) {$\cdots$};
    }
\node[vertex,label=above:{\small $T(k+1)$}] at (6,1) (T1k) {};
\node[vertex,label=below:{\small $T'(k+1)$}] at (6,0) (T2k) {};
\draw (T1k)--(T2k);
\end{tikzpicture}
\end{center}


\item \label{item_case_leaf} If none of the vertices in $S$ are adjacent to $a$ in $G$, then $S$ is a subtree of $b$ in $T'$ and $a$ is not in the path from $r_{T'}$ to $v_{\lambda(T)}$. Then $\lambda(T')=\lambda(T)-1$. Moreover, the $ab$-rotation determines the adjacency of the pair $T(i),T'(i)$ for every $i\in\{0,\ldots,l-1\}$, the adjacency of the pair $T(i+1),T'(i)$ for every $i\in\{l,\ldots,k\}$. Note that this yields a $5$-cycle $T(l-1),T(l),T(l+1),T'(l),T'(l-1)$.
\medskip
        
\begin{center}
\begin{tikzpicture}[x=1.7cm]
\foreach \x/\i in {0/0,2/l-1,3/l,4/l+1}{
    \draw (\x,1) -- (\x+1,1);
    \node[vertex,label=above:{\footnotesize $T(\i)$}] at (\x,1) (T1\x) {};
    }
\node[vertex,label=above:{\footnotesize $T(k+1)$}] at (6,1) (T1k) {};
\node[vertex,label=below:{\footnotesize $T'(k)$}] at (6,0) (T2k) {};
\foreach \x/\i in {0/0,2/l-1,4/l}{
    \node[vertex,label=below:{\footnotesize $T'(\i)$}] at (\x,0) (\x) {};
    \draw (T1\x) -- (\x);
    }
\draw (0,0)--(1,0)--(2,0)--(4,0)--(5,0)--(6,0)--(6,1)--(5,1);
\draw (1,1)--(2,1);

\node[blank] at (1,1) {$\cdots$};
\node[blank] at (1,0) {$\cdots$};
\node[blank] at (5,0) {$\cdots$};
\node[blank] at (5,1) {$\cdots$};
\end{tikzpicture}
\end{center}   


    \end{enumerate}
\item Finally, suppose $b=v_{\lambda(T)}$. 
    \begin{enumerate}
        \item If $a\in K$, then $\lambda(T)=\lambda(T')$ and $v_{\lambda (T')}=a$. We have that for every $i\in\{0,\ldots,k-1\}$, $T(i)$ and $T'(i)$ are adjacent as well as $T(k+1)$ and $T'(k+1)$, since the differ by an $ab$-rotation. Notice that $a$ and $b$ are not adjacent in $T(k)$ as well as in $T'(k)$, and so $T(k)$ and $T'(k)$ are not adjacent in $\R(G_K)$.

        \begin{center}
\begin{tikzpicture}[x=1.5cm]
\foreach \x/\i in {0/0,2/k-1,3/k}{
    \node[vertex,label=above:{\small $T(\i)$}] at (\x,1) (T\x) {};
    }
    \node[vertex,label=above:{\small $T(k+1)$}] at (4,1) (T4) {};
\foreach \x/\i in {0/0,2/k-1,3/k}{
    \node[vertex,label=below:{\small $T'(\i)$}] at (\x,0) (T'\x) {};
    \node[vertex,label=below:{\small $T'(k+1)$}] at (4,0) (T'4) {};
    }    
\foreach \x in {1}{
    \node[blank] at (\x,1) (T\x) {$\cdots$};
    \node[blank] at (\x,0) (T'\x) {$\cdots$};
    }
\foreach \x/\s in {0/1,1/2,2/3,3/4}{
\draw (T\x)--(T\s);
}
\foreach \x/\s in {0/1,1/2,2/4}{
\draw (T'\x)--(T'\s);
}
\foreach \x in {0,2,4}{
\draw (T\x)--(T'\x);
}
\end{tikzpicture}
\end{center}

        \item If $a\notin K$, then $|P^x(T)|=k+2$ and $|P^x(T')|=k+1$. Moreover, we have that the pairs $T(i),T'(i)$ for $i\in\{0,\ldots,k-1\}$ and $T(k+1),T'(k)$ are adjacent.

\begin{center}
\begin{tikzpicture}[x=1.5cm]
\foreach \x/\i in {0/0,2/k-1,3/k}{
    \node[vertex,label=above:{\small $T(\i)$}] at (\x,1) (T\x) {};
    }
    \node[vertex,label=above:{\small $T(k+1)$}] at (4,1) (T4) {};
\foreach \x/\i in {0/0,2/k-1}{
    \node[vertex,label=below:{\small $T'(\i)$}] at (\x,0) (T'\x) {};
    \node[vertex,label=below:{\small $T'(k)$}] at (4,0) (T'4) {};
    }    
\foreach \x in {1}{
    \node[blank] at (\x,1) (T\x) {$\cdots$};
    \node[blank] at (\x,0) (T'\x) {$\cdots$};
    }
\foreach \x/\s in {0/1,1/2,2/3,3/4}{
\draw (T\x)--(T\s);
}
\foreach \x/\s in {0/1,1/2,2/4}{
\draw (T'\x)--(T'\s);
}
\foreach \x in {0,2,4}{
\draw (T\x)--(T'\x);
}
\end{tikzpicture}
\end{center}
\end{enumerate}
    \end{enumerate}
\item Note that the case $a\in A$, $b\notin A$ is covered in item \ref{item_case_leaf} and that the case $a\notin A$, $b\in A$ never occurs.
\end{enumerate}

Let $e=TT'\in E(\R(G))$. Denote $\varepsilon^x(e)$ the set of edges in $\R(G_K)$ with one endpoint in $P^x(T)$ and the other in $P^x(T')$ described in the previous items. 

Let us see that the edges described above are all the edges of $\R(G_K)$. To that end, consider $R\in P^x(T)$ and the search tree $R_\theta$ obtained from $R$ by a $\theta$-rotation with $\theta\in E(R)$. 

As we have mentioned, if $x\in\theta$, then the edge in $\R(G_K)$ determined by $\theta$ is one of the edges of the path determined by $P^x(R)$. On the other hand, if $x\notin\theta$, then $\theta\in E(T)$ and the $\theta$-rotation could be applied to $T$. This rotation determines an edge $TT'$ in $\R(G)$. Then, it follows that $R_\theta\in P^x(T')$, since $\theta$ coincides with one of the rotations previously described in items 1, 2 or 3. 

Thus, we obtain the following characterization of the edges of $\R(G_K)$.

\begin{prop}
\label{prop_edges_of_RGK}
Let $G$ be a graph and $v\in V(G)$. Then,
\begin{enumerate}[(a)]
    \item for every $T\in V(\R(G))$, $P^x(T)$ induces a path in $\R(G_K)$, and
    \item \label{item_edges} $E(\R(G_K)) = \left(\bigcup\limits_{T\in V(\R(G))} E(P^x(T))\right) \cup \left(\bigcup\limits_{e\in E(\R(G))} \varepsilon^x(e)\right)$.
\end{enumerate}
\end{prop}

In the following proposition, we show that there are two induced subgraphs of $\R(G_K)$ isomorphic to $\R(G)$. One of them has the set of all search trees on $G_K$ with root $x$ as its vertex set, and the other has the set of all search trees on $G_K$ that has $x$ as a child of $v_\lambda(T)$.

\begin{prop}
\label{prop_copies_of_RG_in_RGK}
Let $\R_0=\{T(0)\in V(\R(G_K)) \mid T\in V(\R(G))\}$ and $\R_1=\{T(\lambda(T)+1)\in V(\R(G_K)) \mid T\in V(\R(G))\}$. Then $\R_0$ and $\R_1$ induce subgraphs in $\R(G_K)$ that are isomorphic to $\R(G)$.
\end{prop}

\begin{proof}
The maps $f_0: V(\R(G))\rightarrow V(\R_0)$ defined by $f_0(T)=T(0)$ and $f_1:V(\R(G))\rightarrow\R_1$ defined by $f_1(T)=T(\lambda(T)+1)$ are isomorphisms. 
\end{proof}

\begin{figure} 
    \centering
    \begin{subfigure}{0.2\textwidth}
    \begin{tikzpicture}
    \node at (-0.7,0.5) (G) {$G$};
    \node[vertex,label=below:{\small 1}] at (0,0) (1) {};
    \node[vertex,label=above:{\small 2}] at (0,1) (2) {};
    \node[vertex,label=above:{\small 3}] at (1,1) (3) {};
    \draw (1)--(2)--(3)--(1);
\end{tikzpicture}
\medskip 

\begin{tikzpicture}[x=1cm,y=0.8cm]
    \node[vertex,label=above:{\scriptsize 123}] at (0.5,5) (4123) {};
    \node[vertex,label=left:{\scriptsize 213}] at (0,4) (4213) {};  
    \node[vertex,label=right:{\scriptsize 132}] at (1,3) (4132) {};
    \node[vertex,label=left:{\scriptsize 231}] at (0,2) (4231) {};
    \node[vertex,label=right:{\scriptsize 312}] at (1,1) (4312) {};
    \node[vertex,label=below:{\scriptsize 321}] at (0.5,0) (4321) {};
    \draw (4123)--(4213)--(4231)--(4321)--(4312)--(4132)--(4123);
\end{tikzpicture}
\caption{$\R(G)$}
\label{fig_G}
    \end{subfigure}
    \begin{subfigure}{0.6\textwidth}
    \centering
\begin{tikzpicture}
    \node at (-0.7,0.5) (Gv) {$G_K$};
    \node[vertex,label=below:{\small 1}] at (0,0) (1) {};
    \node[vertex,label=above:{\small 2}] at (0,1) (2) {};
    \node[vertex,label=above:{\small 3}] at (1,1) (3) {};
    \node[vertex,label=below:{\small $x=4$}] at (1,0) (4) {};
    \draw (1)--(2)--(3)--(4)--(1)--(3);
    \draw (2)--(4);
\end{tikzpicture}
\medskip

\begin{tikzpicture}[x=1cm,y=0.8cm] 
    \node[vertex,label=above:{\scriptsize 4123}] at (0.5,5) (4123) {};
    \node[vertex,label=above:{\scriptsize 1423}] at (2.5,5) (1423) {};
    \node[vertex,label=above:{\scriptsize 1243}] at (4.5,5) (1243) {};
    \node[vertex,label=above:{\scriptsize 1234}] at (6.5,5) (1234) {};
    \node[vertex,label=left:{\scriptsize 4213}] at (0,4) (4213) {};
    \node[vertex,label=above:{\scriptsize 2413}] at (2,4) (2413) {};
    \node[vertex,label=below:{\scriptsize 2143}] at (4,4) (2143) {};
    \node[vertex,label=right:{\scriptsize 2134}] at (6,4) (2134) {};
    \node[vertex,label=left:{\scriptsize 4132}] at (1,3) (4132) {};
    \node[vertex,label=below:{\scriptsize 1432}] at (3,3) (1432) {};
    \node[vertex,label=above:{\scriptsize 1342}] at (5,3) (1342) {};
    \node[vertex,label=right:{\scriptsize 1324}] at (7,3) (1324) {};
    \node[vertex,label=left:{\scriptsize 4231}] at (0,2) (4231) {};
    \node[vertex,label=below:{\scriptsize 2431}] at (2,2) (2431) {};
    \node[vertex,label=above:{\scriptsize 2341}] at (4,2) (2341) {};
    \node[vertex,label=right:{\scriptsize 2314}] at (6,2) (2314) {};
    \node[vertex,label=left:{\scriptsize 4312}] at (1,1) (4312) {};
    \node[vertex,label=above:{\scriptsize 3412}] at (3,1) (3412) {};
    \node[vertex,label=below:{\scriptsize 3142}] at (5,1) (3142) {};
    \node[vertex,label=right:{\scriptsize 3124}] at (7,1) (3124) {};
    \node[vertex,label=below:{\scriptsize 4321}] at (0.5,0) (4321) {};
    \node[vertex,label=below:{\scriptsize 3421}] at (2.5,0) (3421) {};
    \node[vertex,label=below:{\scriptsize 3241}] at (4.5,0) (3241) {};
    \node[vertex,label=below:{\scriptsize 3214}] at (6.5,0) (3214) {};
    \draw (4123)--(4213)--(4231)--(4321)--(4312)--(4132)--(4123);
    \draw (4123)--(1423)--(1243)--(1234);
    \draw (2413)--(2431);
    \draw (3421)--(3412);
    \draw (1432)--(1423);
    \draw (1243)--(2143);
    \draw (2341)--(3241);
    \draw (3142)--(1342);
    \draw (1234)--(2134)--(2314)--(3214)--(3124)--(1324)--(1234);
    \draw (4213)--(2413)--(2143)--(2134);
    \draw (4132)--(1432)--(1342)--(1324);
    \draw (4231)--(2431)--(2341)--(2314);
    \draw (4312)--(3412)--(3142)--(3124);
    \draw (4321)--(3421)--(3241)--(3214);
\end{tikzpicture}
    \caption{$\R(G_K)$ for $K=\{1,2,3\}$}
    \label{fig_GK}
    \end{subfigure}
    \caption{ }
\end{figure}

Figures \ref{fig_G} and \ref{fig_GK} depicts the graph $G=K_3$, and $G_K$ for the clique $K=\{1,2,3\}$, where the added simplicial vertex $x$ is relabelled 4, to make clear that $G_K$ is isomorphic to $K_4$.

\subsection{Adding a true twin vertex}
\label{subsection_adding_tt}

Let $G$ be a connected graph and $v\in V(G)$. We denote by $G_v$ the graph determined by
\begin{align*}
    V(G_v) &= V(G)\cup \{v'\},\\
    E(G_v) &= E(G) \cup \{uv' \mid u\in N_G(v)\} \cup \{vv'\}.
\end{align*}
That is, $G_v$ is the graph obtained from $G$ by adding a vertex $v'$ to $G$ that is a true twin of $v$. 

In order to define search trees of $G_v$ in terms of search trees on $G$, we introduce the following function. Given $u,v\in V(G)$, denote $\rho_{u,v}: V(G)\rightarrow V(G)$ the function defined by $\rho_{u,v}(u)=v$, $\rho_{u,v}(v)=u$ and $\rho_{u,v}(x)=x$ if $x\in V(G)-\{u,v\}$. Notice that if $u$ and $v$ are true twins in $G$, then $\rho_{u,v}$ is a graph isomorphism. 

Let $G,H$ be connected graphs and let $f:V(G)\rightarrow V(H)$ be a graph isomorphism. For every $T\in V(\R(G))$, we denote by $f^{\ast}(T)$ the rooted tree with vertex set $f(V(G))=V(H)$ and edge set $\{f(x)f(y) \mid xy\in E(T)\}$. Notice that since $f$ is a graph isomorphism, $f^{\ast}(T)$ is a search tree on $H$.

Let $G$ be a connected graph and $v\in V(G)$. Consider $T$ a search tree on $G$. For $i\in\{0,\ldots,d_{T,v}\}$ and $j\in\{1,2\}$, we denote $T(i,j)$ the rooted tree with vertices $V(G_v)$ defined as follows. 

\begin{itemize}
    \item $T(i,1)$ is the rooted tree $T(i,v',v)$.
    \item $T(i,2)=\rho^\ast_{v,v'}(T(i,1))$.
\end{itemize}

Additionally, we denote 
\[P(T) = \{T(i,j) \mid i\in\{0,\ldots,d_{T,v}\}, j\in\{1,2\}\}.
\]

\begin{figure}
    \centering
\begin{tikzpicture}[x=0.7cm,y=0.8cm] 
    \node[vertex,label=above:{\scriptsize 4123}] at (0.5,5) (4123) {};
    \draw (0.5,4) node[rectangle,minimum height=0.3cm,fill=white, fill opacity=0.8] { };
    \node[vertex,label=above:{\scriptsize 1423}] at (2.5,5) (1423) {};
    \node[vertex,fill=red,label=above:{\scriptsize 1243}] at (4.5,5) (1243) {};
    \node[vertex,fill=red,label=above:{\scriptsize 1234}] at (6.5,5) (1234) {};
    \node[vertex,label=above:{\scriptsize 1324}] at (8.5,5) (1324) {};
    \node[vertex,label=above:{\scriptsize 3124}] at (10.5,5) (3124) {};
    \node[vertex,label=left:{\scriptsize 4213}] at (0,4) (4213) {};
    \node[vertex,label=above:{\scriptsize 2413}] at (2,4) (2413) {};
    \node[vertex,fill=blue,label=below:{\scriptsize 2143}] at (4,4) (2143) {};
    \node[vertex,fill=blue,label=below:{\scriptsize 2134}] at (6,4) (2134) {};
    \node[vertex,label=above:{\scriptsize 2314}] at (8,4) (2314) {};
    \node[vertex,label=right:{\scriptsize 3214}] at (10,4) (3214) {};
    \node[vertex,label=left:{\scriptsize 4132}] at (1,3) (4132) {};
    \node[vertex,label=below:{\scriptsize 1432}] at (3,3) (1432) {};
    \node[vertex,label=below:{\scriptsize 1342}] at (9,3) (1342) {};
    \node[vertex,label=right:{\scriptsize 3142}] at (11,3) (3142) {};
    \node[vertex,label=left:{\scriptsize 4231}] at (0,2) (4231) {};
    \node[vertex,label=below:{\scriptsize 2431}] at (2,2) (2431) {};
    \node[vertex,label=below:{\scriptsize 2341}] at (8,2) (2341) {};
    \node[vertex,label=right:{\scriptsize 3241}] at (10,2) (3241) {};
    \node[vertex,label=left:{\scriptsize 4312}] at (1,1) (4312) {};
    \node[vertex,label=right:{\scriptsize 3412}] at (11,1) (3412) {};
    \node[vertex,label=below:{\scriptsize 4321}] at (0.5,0) (4321) {};
    \node[vertex,label=below:{\scriptsize 3421}] at (10.5,0) (3421) {};
    \draw (4123)--(4213)--(4231)--(4321)--(4312)--(4132)--(4123);
    \draw (3124)--(3214)--(3241)--(3421)--(3412)--(3142)--(3124);
    \draw (4123)--(1423)--(1243);
    \draw[red] (1243)--(1234);
    \draw (1234)--(1324)--(3124);
    \draw (4213)--(2413)--(2143);
    \draw[blue] (2143)--(2134);
    \draw (2134)--(2314)--(3214);
    \draw (4132)--(1432)--(1342)--(3142);
    \draw (4231)--(2431)--(2341)--(3241);
    \draw (4312)--(3412);
    \draw (4321)--(3421);
    \draw (1243)--(2143);
    \draw (1234)--(2134);
    \draw (2413)--(2431);
    \draw (2314)--(2341);
    \draw (1432)--(1423);
    \draw (1342)--(1324);
\end{tikzpicture}
\caption{$\R(K_4)$}
\label{fig:RK4}
\end{figure}

Figure \ref{fig:P(T)} depicts a graph $G$ and $G_v$ for $v\in V(G)$, as well as a search tree $T$ on $G$ and the members of $P(T)$. Notice that in this case, $P(T)\subseteq V(\R(G_v))$. This is true in general and moreover, the sets $P(T)$ with $T\in V(\R(G))$ are a partition of $V(\R(G_v))$, as we state in the following proposition. 

\begin{figure} 
\begin{center}
\begin{tikzpicture}
\node[vertex,label=left:{\small $b$}] at (0,0) (b1) {};
\node[vertex,label=left:{\small $a$}] at (0,1) (a1) {};
\node[vertex,label=right:{\small $c$}] at (1,0) (c1) {};
\node[vertex,label=right:{\small $v$}] at (1,1) (v1) {};
\draw (a1)--(b1)--(c1)--(v1)--(a1);
\node at (0.5,-0.8) (G) {\small $G$};

\node[vertex,label=left:{\small $b$}] at (3,0) (b2) {};
\node[vertex,label=left:{\small $a$}] at (3,1) (a2) {};
\node[vertex,label=right:{\small $c$}] at (4,0) (c2) {};
\node[vertex,label={[xshift=0.3cm, yshift=-0.4cm]\small $v$}] at (4,1) (v12) {};
\node[vertex,label=right:{\small $v'$}] at (5,1.5) (v22) {};
\draw (a2)--(b2)--(c2)--(v12)--(a2);
\draw (a2)--(v22)--(c2);
\draw (v12)--(v22);
\node at (4,-0.8) (Gv) {\small $G_v$};

\node[vertex,label=left:{\small $b$}] at (7,0) (b3) {};
\node[vertex,label=right:{\small $a$}] at (7.5,1.5) (a3) {};
\node[vertex,label=right:{\small $c$}] at (7.5,0.75) (c3) {};
\node[vertex,label=right:{\small $v$}] at (8,0) (v3) {};
\draw (a3)--(c3)--(b3);
\draw (c3)--(v3);
\node at (7.5,-0.8) (G) {\small $T$};
\end{tikzpicture}
\medskip
\end{center}

\begin{center} 
\begin{tikzpicture}[scale=0.9]
\node[vertex,label=right:{\small $v$}] at (0,3) (v11) {};
\node[vertex,label=right:{\small $a$}] at (0,2) (a1) {};
\node[vertex,label=right:{\small $c$}] at (0,1) (c1) {};
\node[vertex,label=left:{\small $b$}] at (-0.5,0) (b1) {};
\node[vertex,label=right:{\small $v'$}] at (0.5,0) (v21) {};
\draw (v11)--(a1)--(c1)--(b1);
\draw (c1)--(v21);
\node at (0,-0.8) (T01) {\small $T(0,1)$};

\node[vertex,label=right:{\small $a$}] at (3,3) (a2) {};
\node[vertex,label=right:{\small $v$}] at (3,2) (v12) {};
\node[vertex,label=right:{\small $c$}] at (3,1) (c2) {};
\node[vertex,label=left:{\small $b$}] at (2.5,0) (b2) {};
\node[vertex,label=right:{\small $v'$}] at (3.5,0) (v22) {};
\draw (a2)--(v12)--(c2)--(v22);
\draw (c2)--(b2);
\node at (3,-0.8) (T11) {\small $T(1,1)$};

\node[vertex,label=right:{\small $a$}] at (6,3) (a3) {};
\node[vertex,label=right:{\small $c$}] at (6,2) (c3) {};
\node[vertex,label=left:{\small $b$}] at (5.5,1) (b3) {};
\node[vertex,label=right:{\small $v$}] at (6.5,1) (v13) {};
\node[vertex,label=right:{\small $v'$}] at (6.5,0) (v23) {};
\draw (a3)--(c3)--(v13)--(v23);
\draw (c3)--(b3);
\node at (6,-0.8) (T21) {\small $T(2,1)$};

\node[vertex,label=right:{\small $a$}] at (9,3) (a4) {};
\node[vertex,label=right:{\small $c$}] at (9,2) (c4) {};
\node[vertex,label=left:{\small $b$}] at (8.5,1) (b4) {};
\node[vertex,label=right:{\small $v'$}] at (9.5,1) (v24) {};
\node[vertex,label=right:{\small $v$}] at (9.5,0) (v14) {};
\draw (a4)--(c4)--(v24)--(v14);
\draw (c4)--(b4);
\node at (9,-0.8) (T22) {\small $T(2,2)$};

\node[vertex,label=right:{\small $a$}] at (12,3) (a5) {};
\node[vertex,label=right:{\small $v'$}] at (12,2) (v25) {};
\node[vertex,label=right:{\small $c$}] at (12,1) (c5) {};
\node[vertex,label=left:{\small $b$}] at (11.5,0) (b5) {};
\node[vertex,label=right:{\small $v$}] at (12.5,0) (v15) {};
\draw (a5)--(v25)--(c5)--(v15);
\draw (c5)--(b5);
\node at (12,-0.8) (T12) {\small $T(1,2)$};

\node[vertex,label=right:{\small $v'$}] at (15,3) (v26) {};
\node[vertex,label=right:{\small $a$}] at (15,2) (a6) {};
\node[vertex,label=right:{\small $c$}] at (15,1) (c6) {};
\node[vertex,label=left:{\small $b$}] at (14.5,0) (b6) {};
\node[vertex,label=right:{\small $v$}] at (15.5,0) (v16) {};
\draw (v26)--(a6)--(c6)--(b6);
\draw (c6)--(v16);
\node at (15,-0.8) (T02) {\small $T(0,2)$};
\end{tikzpicture}
\end{center}
    \caption{Members of $P(T)$ for the search tree $T$ on $G$.}
    \label{fig:P(T)}
\end{figure}

\begin{prop}
Let $G$ be a connected graph and let $v\in V(G)$. Then $\{P(T)\mid T\in V(\R(G))\}$ is a partition of $V(\R(G_v))$.
\end{prop}

\begin{proof}
Let us first see that 
\[
V(\R(G_v)) = \bigcup_{T\in V(\R(G))} P(T).
\]

Let $T\in V(\R(G))$. Recall that, by Lemma \ref{obs_Tix_are_ST_and_path}, for $i\in\{0,\dots,d_{T,v}\}$, $T(i,1)$ is a search tree on $G_v$, and since $\rho_{v,v'}$ is a graph isomorphism, $T(i,2)$ is also a search tree on $G_v$. Therefore, we have
\[
V(\R(G_v)) \supseteq \bigcup_{T\in V(\R(G))} P(T).
\]

On the other hand, let $T'$ be a search tree on $G_v$. From Remark \ref{obs_uv_branch}, we known that $v$ and $v'$ are in the same branch of $T'$. Assume that $v$ is a descendant of $v'$ in $T'$. Let $T=p(T',v')$. Then $T$ is a search tree on $G$ as a consequence of Lemma \ref{obs_pTx_is_ST} and, moreover $T'=T(i,1)$ for $i=d_{T',v'}$. 

Thus, $V(\R(G_v)) \subseteq \bigcup\limits_{T\in V(\R(G))} P(T)$.

Finally, let $T,T'$ be search trees on a connected graph $G$ such that $T(i,j)=T'(i',j')$ for some $i\in\{0,\ldots,d_{T,v}\}$, $i'\in\{0,\ldots,d_{T',v}\}$, $j,j'\in\{1,2\}$. Then:
\begin{itemize}
    \item $j=j'$ since if it were $j\neq j'$, then $v$ would be descendant of $v'$ in one tree and $v'$ would be descendant of $v$ in the other; and
    \item $i=d_{T(i,j),w}=d_{T'(i',j),w}=i'$, where $w=v'$ if $j=1$ and $w=v$ if $j=2$.
\end{itemize}
Thus, since $T(i,j)=T'(i,j)$, by Remark \ref{obs_ins_and_prun_are_inverses} $T=T'$, and this completes the proof.
\end{proof}




As in the previous section, now we describe the edges in $\R(G_v)$. 

First notice that for every $T\in V(\R(G))$, $i\in \{0,\ldots,d_{T,v}-1\}$, and $j\in\{1,2\}$, $T(i,j)$ and $T(i+1,j)$ are adjacent in $\R(G_v)$ by Lemma $\ref{obs_Tix_are_ST_and_path}$. Additionally, $T(d_{T,v},1)$ and $T(d_{T,v},2)$ are adjacent since they differ by a $vv'$-rotation. Thus, the search trees on $P(T)$ are on a path in $\R(G_v)$.

Now, let $T,T'$ be adjacent vertices of $\R(G)$. We study adjacencies in $\R(G_v)$ with one endpoint in $P(T)$ and the other in $P(T')$. To this effect, we analyse the possible types of rotations that determine adjacencies between $T$ and $T'$.

Let $T$ be a search tree on $G$ and let $A$ be the path in $T$ between $r_T$ and $v$. Let $k=d_{T,v}$. Let $T'$ be a search tree on $G$ obtained from $T$ by rotating a pair of adjacent vertices $a,b$. Suppose $a$ is the parent of $b$ in $T$. 

\begin{enumerate}
\item Assume $a$ and $b$ are not in $A$. Then, $r_{T'}=r_T$ and the path in $T'$ between $r_{T'}$ and $v$ is exactly $A$. Thus, $|P(T)|=|P(T')|$ and for every $i\in \{0,\ldots,k\}$ and $j\in\{1,2\}$, $T(i,j)$ is adjacent to $T'(i,j)$, since they differ by the rotation of the pair $a,b$. Moreover, note that $T(i,j)$ and $T'(i',j')$ are adjacent if and only if $i=i'$ and $j=j'$.

\begin{center}
\begin{tikzpicture}[x=1.5cm]
\foreach \x/\i/\j in {0/0/1,1/1/1,3/k/1,4/k/2}{
    \node[vertex,label=above:{$T(\i,\j)$}] at (\x,1) (T1\x) {\tiny};
    \node[vertex,label=below:{$T'(\i,\j)$}] at (\x,0) (T2\x) {};
    \draw (\x,1) -- (\x+1,1);
    \draw (\x,0) -- (\x+1,0);
    \draw (\x,0) -- (\x,1);
    }
\foreach \x in {2,5}{
    \draw (\x,1) -- (\x+1,1);
    \draw (\x,0) -- (\x+1,0);
    \node[blank] at (\x,1) (T1\x) {$\cdots$};
    \node[blank] at (\x,0) (T2\x) {$\cdots$};
    }
\foreach \y/\d/\u in {0/'/below,1/ /above}{
    \node[vertex,label=\u:{$T\d (0,2)$}] at (6,\y) (T6\y) {};
    }
    \draw (6,0) -- (6,1);
\end{tikzpicture}
\end{center}
\item Now we consider the case that $a$ and $b$ are both in $A$. Then $a\neq v$, since $a$ is the parent of $b$ in $T$. 
    \begin{enumerate}
    \item Suppose $b\neq v$. Let $S$ be the subtree of $b$ in $T$ that contains $v$. Suppose $A$ has vertices $a_0=r_T,a_1,\ldots,a_k=v$ and let $l\in\{1,\ldots,k-1\}$ be such that  $a=a_{l-1}$ and $b=a_l$. 
        \begin{enumerate}
        \item If there is a vertex in $S$ that is adjacent to $a$ in $G$, then $S$ is a subtree of $a$ in $T'$. Thus, $a$ remains in the path from $r_{T'}$ to $v$ in $T'$. Therefore $d_{T,v}=d_{T',v}$ and for every $i\in\{0,\ldots,k\}$ such that $i\neq l$ and $j\in\{1,2\}$, $T(i,j)$ is adjacent to $T'(i,j)$ since they differ by the rotation of the pair $a,b$. Also, for $j\in\{1,2\}$,  $T(l,j)$ and $T'(l,j)$ are nonadjacent since $a$ and $b$ are nonadjacent in those search trees.
        \begin{center}
\begin{tikzpicture}[x=1.5cm]
\foreach \x/\i/\c in {0/0/black,2/l-1/black,3/l/white,4/l+1/black,6/k/black}{
    \draw (\x,1) -- (\x+1,1);
    \draw (\x,0) -- (\x+1,0);
    \draw[\c] (\x,0) -- (\x,1);
    \node[vertex,label=above:{\small $T(\i,1)$}] at (\x,1) (T1\x) {};
    \node[vertex,label=below:{\small $T'(\i,1)$}] at (\x,0) (T2\x) {};
    }
\foreach \x in {1,5}{
    \draw (\x,1) -- (\x+1,1);
    \draw (\x,0) -- (\x+1,0);
    \node[blank] at (\x,1) (T1\x) {$\cdots$};
    \node[blank] at (\x,0) (T2\x) {$\cdots$};
    }
    \node[blank] at (7,1) (x) {$\cdots$};
    \node[blank] at (7,0) (y) {$\cdots$};
\end{tikzpicture}
\end{center}
        \item \label{item_case} If none of the vertices in $S$ are adjacent to $a$ in $G$, then $S$ is a subtree of $b$ in $T'$ and $a$ is not in the path from $r_{T'}$ to $v$. Then $d_{T',v}=d_{T,v}-1$. Moreover, the rotation of $a$ and $b$ determines the adjacency of the pairs $T(i,j),T'(i,j)$ for $i\in\{0,\ldots,l-1\}$, $j\in\{1,2\}$ and the adjacency of the pairs $T(i+1,j),T'(i,j)$ for $i\in\{l,\ldots,k-1\}$, $j\in\{1,2\}$. Note that the vertices $T(l-1,j),T(l,j),T(l+1,j),T'(l,j),T'(l-1,j)$ induces a $5$-cycle, for each $j\in\{1,2\}$.
        \medskip
        
\begin{center}
\begin{tikzpicture}[x=1.7cm]
\foreach \x/\i in {0/0,2/l-1,3/l,4/l+1,5/l+2}{
    \draw (\x,1) -- (\x+1,1);
    \node[vertex,label=above:{\footnotesize $T(\i,1)$}] at (\x,1) (T1\x) {};
    }
\foreach \x/\i in {0/0,2/l-1,4/l,5/l+i}{
    \node[vertex,label=below:{\footnotesize $T'(\i,1)$}] at (\x,0) (\x) {};
    \draw (T1\x) -- (\x);
    }
\draw (0,0)--(1,0)--(2,0)--(4,0)--(5,0)--(6,0);
\draw (1,1)--(2,1);
\node[blank] at (1,1) {$\cdots$};
\node[blank] at (1,0) {$\cdots$};
\node[blank] at (6,1) {$\cdots$};
\node[blank] at (6,0) {$\cdots$};
\end{tikzpicture}
\end{center}
        \end{enumerate}
    \item Finally, suppose $b=v$. Then $|P(T)|=2(k+1)$ and $|P(T')|=2k$. Moreover, we have adjacencies of the pairs $T(i,j),T'(i,j)$ for $i\in\{0,\ldots,k-1\}$, $j\in\{1,2\}$.
    \end{enumerate}
\begin{center}
\begin{tikzpicture}[x=1.5cm]
\foreach \x/\i/\j in {0/0/1,2/k-1/1,3/k/1,4/k/2,5/k-1/2,7/0/2}{
    \node[vertex,label=above:{\small $T(\i,\j)$}] at (\x,1) (T\x) {};
    }
\foreach \x/\i/\j in {0/0/1,2/k-1/1,5/k-1/2,7/0/2}{
    \node[vertex,label=below:{\small $T'(\i,\j)$}] at (\x,0) (T'\x) {};
    }    
\foreach \x in {1,6}{
    \node[blank] at (\x,1) (T\x) {$\cdots$};
    \node[blank] at (\x,0) (T'\x) {$\cdots$};
    }
\foreach \x/\s in {0/1,1/2,2/3,3/4,4/5,5/6,6/7}{
\draw (T\x)--(T\s);
}
\foreach \x/\s in {0/1,1/2,2/5,5/6,6/7}{
\draw (T'\x)--(T'\s);
}
\foreach \x in {0,2,5,7}{
\draw (T\x)--(T'\x);
}
\end{tikzpicture}
\end{center}

\item Note that the case $a\in A$, $b\notin A$ is covered in item \ref{item_case} and that the case $a\notin A$, $b\in A$ never occurs.
\end{enumerate}

Let $e=\{T,T'\}$ $\in E(\R(G))$. We denote $\varepsilon_e$ the set of edges with one endpoint in $P(T)$ and the other in $P(T')$ enumerated above.

Let us see that the edges previously described are all the edges of $\R(G_v)$. Consider $T(i,j)$ and the search tree $T_\theta$ obtained from $T(i,j)$ by a $\theta$-rotation, where $\theta\in E(T(i,j))$. 

Notice that if $v\in\theta$ or $v'\in\theta$, then $\theta$ determines an edge in $\R(G_v)$ that is contained in the path $P(T)$.

So, let us consider the case where $\theta$ does not involve neither $v$ nor $v'$. Then $\theta\in E(T)$ and the $\theta$-rotation on $T$ determines an edge $TT'$ in $\R(G)$. It follows that $T_\theta\in P(T')$, since $\theta$ coincides with one of the rotations previously described in item 1 or 2(a) or 3. Further notice that $T_\theta=T'(i',j)$ for some $i'\in\{i-1,i,i+1\}$.

Finally, suppose $\theta=uw$, where $w=v'$ if $j=1$ and $w=v$ if $j=2$, and that $u\notin \{v,v'\}$. Let us consider $T'$ obtained from $T$ by the rotation $uv$ in $\R(G)$. Notice that $T(i,j)T_\theta$ is one of the adjacencies described in item 2(b). Thus, $T_\theta=T'(i,j)$.

Therefore, we have the following result describing the edges of $\R(G_v)$.

\begin{prop}
\label{prop_edges_of_RGv}
Let $G$ be a connected graph and let $v\in V(G)$. Then 
\begin{enumerate}[(a)]
    \item for every $T\in\R(G)$, $P(T)$ induces a path in $\R(G_v)$, and
    \item \label{item_edges_Gv} $E(\R(G_v)) =\left(\bigcup\limits_{T\in V(\R(G))} E(P(T)) \right) \cup \left(\bigcup\limits_{e\in E(\R(G))} \varepsilon_e \right)$.
\end{enumerate}
\end{prop}

We hereby present two results related to the structure elucidated in $\R(G_v)$. These properties will be utilized in subsequent sections to derive results related to vertex pair distances within the graph as well as its chromatic number.

\begin{prop}
\label{prop_copies_of_AG}
Let $G$ be a connected graph and $v\in V(G)$. Let $j\in\{1,2\}$. Then the subgraph of $\R(G_v)$ induced by the set $\{T(0,j)\mid T\in V(\R(G))\}$ is isomorphic to $\R(G)$.    
\end{prop}

\begin{proof}
The graph isomorphism is given by $f:V(\R(G))\rightarrow \{T(0,j)\mid T\in V(\R(G))\}$, where $f(T)=T(0,j)$.
\end{proof}

\begin{obs}
\label{obs_symmetry_in_A(Gv)}
Let $G$ be a connected graph and $v\in V(G)$. Let $j\in\{1,2\}$. Let $\A_j$ be the subgraph of $\R(G_v)$ induced by the set of vertices $\{T(i,j) \mid T\in V(\R(G)), i\in\{0,\ldots,d_{T,v}\}\}$. Then
\begin{enumerate}[(a)]
    \item $V(\R(G_v)) = V(\A_1)\cup V(\A_2)$. 
    \item $\A_1$ and $\A_2$ are isomorphic. The map
    $\rho_{vv'}:V(\A_1)\rightarrow V(\A_2)$ is a graph isomorphism.
    \item\label{item_boundary_edges}
    every edge with one endpoint in $\A_1$ and the other in $\A_2$ is of the form $T(d_{T,v},1)T(d_{T,v},2)$ with $T\in V(\R(G))$.
\end{enumerate}
\end{obs}

\subsection{Deleting a set of edges connecting true twins}

Given a graph $G$ and $v\in V(G)$, we denote  $\wt G_v$ the graph determined by
\begin{align*}
    V(\wt G_v) &= V(G)\cup \{v'\},\\
    E(\wt G_v) &= E(G) \cup \{uv' \mid u\in N_G(v)\}.
\end{align*}
That is, $\wt G_v$ is the graph obtained by adding a false twin to the vertex $v$.

Observe that $\wt G_v$ can be obtained by removing the edge $vv'$ from $G_v$. Therefore, we broadly examine the impact of removing a set of edges connecting vertices from a subset $W \subseteq V(G)$ of true twins in $G$ on the associated rotation graph. Specifically, we describe the structure of $\R(G-S)$ as a quotient graph of $\R(G)$, where $S=E(G[W])$ and $G-S$ denote the graph obtained from $G$ by removing all the edges in $S$.
 
Recall that an equivalence relation $\sim$ (or a partition) on $V(G)$ determines a quotient graph $Q$ of $G$ whose vertex set is $V(Q)=V(G)/\sim$ and where two equivalence classes $[u],[v]$ are adjacent in $Q$ if and only if there exist $x,y$ adjacent in $G$ such that $x\in[u]$ and $y\in[v]$. For graphs $G$ and $H$, we say that a function $q\colon V(G)\rightarrow V(H)$ is a quotient map if for every $uv\in E(G)$, $q(u)q(v)\in E(H)$ or $q(u)=q(v)$. Notice that if $H$ is a quotient graph of $G$ determined by an equivalence relation $\sim$, then 
    \begin{align*}
    q: V(G)&\rightarrow V(H)=V(G)/\sim    \\
    u &\mapsto [u]
    \end{align*}
is a quotient map. Conversely, if $q:V(G)\rightarrow V(H)$ is a quotient map then the graph $q(G)$ determined by $V(q(G))=\{q(v)\mid v\in G\}$ and $E(q(G))=\{q(u)q(v)\mid uv\in E(G), q(u)\neq q(v)\}$ is a quotient graph determined by an equivalence relation $\sim$ given by $u\sim v \iff q(u)=q(v)$.

Thus, in order to show that $\R(G-S)$ is a  quotient graph of $\R(G)$, we define a map $\pi:V(\R(G))\rightarrow V(\R(G-S))$ and show that it is a quotient map. 

Notice that the search trees on $G-S$ and the search trees on $G$ have the same vertex set. In addition, $G$ and $G-S$ have common search trees, but $V(\R(G))\setminus V(\R(G-S))\neq\emptyset$ and $V(\R(G-S))\setminus V(\R(G))\neq\emptyset$. This relation motivates the following definition.
We say that $T\in V(\R(G))$ is $W$-special if there exists a unique subset $L_T$ of $W$ with $|L_T|\geq 2$ and such that $T[L_T]$ is a path with one endpoint adjacent in $T$ to a vertex $q_T$ of $V(G)-W$, and the other endpoint is a leaf of $T$. In addition, we say that  $TT'\in E(\R(G))$ is a $W$-special edge if $T$ and $T'$ are $W$-special search trees on $G$ such that $L_T=L_{T'}$ and the adjacency is determined by a rotation of a pair of vertices from $L_T$. In Figure \ref{fig:exampleSpecial}, $W=\{x_1,x_2,x_3\}$ is a set of true twins in $\SPK_{3,3}$, the search trees $R$ and $S$ on $G$ are $W$-special and $T$ is not $W$-special. In addition, $L_R=L_S=W$, $q_R=q_S=y_3$, $R$ and $S$ are adjacent in $\R(G)$ and $RS$ is a $W$-special edge.

 If for some $T\in V(\R(G-S))$, two or more vertices from $W$ are leaves in $T$, 
then $T\notin V(\R(G))$ by Remark \ref{obs_uv_branch}. Besides, if no pair of vertices from $W$ are leaves in $T$, then $T\in V(\R(G))$ (moreover, $T$ is not $W$-special) since $V(T|u)$ induces a connected component of $G$ and $G-S$. Reciprocally, if $T\in V(\R(G))$ is $W$-special, then $T\notin V(\R(G-S))$, since the vertices in $L_T$ are not adjacent in $G-S$. Otherwise, $T\in V(\R(G-S))$.


Let $S=E(G[W])$ and let $\pi\colon V(\R(G))\rightarrow V(\R(G-S))$ be the map defined by 
\[
\pi(T)=\begin{cases}
        T_{\land}, \quad\text{ if $T$ is a $W$-special search tree,} \\
        T, \quad\text{ otherwise}.
\end{cases}
\]
where $T_{\land}$ is the search tree on $G-S$ such that $T-L_T=T_{\land}-L_T$ and all the vertices of $L_T$ are leaves in $T_{\land}$ with a common parent $q_T$. Observe that $T_{\land}$ can be obtained following the elimination order that determines $T$. In addition, note that by the previous paragraph $\pi$ is well defined and surjective.

In the graph $K_4$, $W=\{3,4\}$ is a set of true twins. Deleting the edge 34 from $K_4$ we obtain the split complete graph $\SPK_{2,2}$. Figure \ref{fig:imagesOfPi} shows search trees $U_1$, $U_2$, $V_1$, $V_2$ on $K_4$ such that $\pi(U_1)=\pi(U_2)$ and $\pi(V_1)=\pi(V_2)$. 

\begin{figure}
    \centering
    \begin{tikzpicture}
    \node[vertex,blue,fill=blue,label=right:{\small $1$}] at (5,3) (11) {};
    \node[vertex,blue,fill=blue,label=right:{\small $2$}] at (5,2) (12) {};
    \node[vertex,blue,fill=blue,label=right:{\small $4$}] at (5,1) (13) {};
    \node[vertex,blue,fill=blue,label=right:{\small $3$}] at (5,0) (14) {};
    \draw (11) [color=blue]--(12)--(13)--(14);
    \node at (5,-0.6)  (U2) {$U_2$};
    
    \node[vertex,,blue,fill=blue,label=left:{\small $1$}] at (1,3) (21) {};
    \node[vertex,,blue,fill=blue,label=left:{\small $2$}] at (1,2) (22) {};
    \node[vertex,,blue,fill=blue,label=left:{\small $3$}] at (1,1) (24) {};
    \node[vertex,,blue,fill=blue,label=left:{\small $4$}] at (1,0) (23) {};
    \draw (21) [color=blue]--(22)--(24)--(23);
    \node at (1,-0.6) (U1) {$U_1$};

    \draw[->] (4.65,1.5) -- (3.85,1.5);
    \node at (4.25,2) (pi) {$\pi$};
    \draw[->] (1.35,1.5) -- (2.15,1.5);
    \node at (1.75,2) (pi) {$\pi$};
    \node[vertex,blue,fill=blue,label=left:{\small $1$}] at (3,2.5) (31) {};
    \node[vertex,blue,fill=blue,label=left:{\small $2$}] at (3,1.5) (32) {};
    \node[vertex,blue,fill=blue,label=left:{\small $3$}] at (2.5,0.5) (33) {};
    \node[vertex,blue,fill=blue,label=left:{\small $4$}] at (3.5,0.5) (34) {};
    \draw (31)[color=blue]--(32)--(33);
    \draw (32)[color=blue]--(34);
    \node at (3,-0.6) (U) {$U$};

        \node[vertex,red,fill=red,label=right:{\small $2$}] at (13,3) (51) {};
    \node[vertex,red,fill=red,label=right:{\small $1$}] at (13,2) (52) {};
    \node[vertex,red,fill=red,label=right:{\small $4$}] at (13,1) (53) {};
    \node[vertex,red,fill=red,label=right:{\small $3$}] at (13,0) (54) {};
    \draw (51)[color=red]--(52)--(53)--(54);
    \node at (13,-0.6) (V2) {$V_2$};
    
    \node[vertex,red,fill=red,label=left:{\small $2$}] at (9,3) (41) {};
    \node[vertex,red,fill=red,label=left:{\small $1$}] at (9,2) (42) {};
    \node[vertex,red,fill=red,label=left:{\small $3$}] at (9,1) (44) {};
    \node[vertex,red,fill=red,label=left:{\small $4$}] at (9,0) (43) {};
    \draw (41)[color=red]--(42)--(44)--(43);
    \node at (9,-0.6) (V1) {$V_1$};

    \draw[->] (12.65,1.5) -- (11.85,1.5);
    \node at (12.25,2) (pi) {$\pi$};
    \draw[->] (9.35,1.5) -- (10.15,1.5);
    \node at (9.75,2) (pi) {$\pi$};
    \node[vertex,red,fill=red,label=left:{\small $2$}] at (11,2.5) (61) {};
    \node[vertex,red,fill=red,label=left:{\small $1$}] at (11,1.5) (62) {};
    \node[vertex,red,fill=red,label=left:{\small $3$}] at (10.5,0.5) (63) {};
    \node[vertex,red,fill=red,label=left:{\small $4$}] at (11.5,0.5) (64) {};
    \draw (61)[color=red]--(62)--(63);
    \draw (62)[color=red]--(64);
    \node at (11,-0.6) (V) {$V$};
    \end{tikzpicture}
    
    \caption{Four search trees on $K_{4}$, and their images through $\pi$ on $\SPK_{2,2}$.}
    \label{fig:imagesOfPi}
\end{figure}

\begin{prop}\label{prop_pi_quotient}
Let $G$ be a connected graph and $W\subseteq V(G)$ a set of true twins. Then $\R(G-S)$ is a quotient graph of $\R(G)$, where $S=E(G[W])$.
\end{prop}

\begin{proof}
We show that the function $\pi$ previously defined is a quotient map, i.e. that for every $TT'\in E(\R(G))$, $\pi(T)\pi(T')\in E(\R(G-S))$ or $\pi(T)=\pi(T')$. 

Let $TT'\in E(\R(G))$. Suppose $T$ is a $W$-special search tree. 

Assume first that $T'$ is also a $W$-special search tree with $L_T=L_{T'}$ and $q_T=q_{T'}$. If $T$ and $T'$ differ by a rotation that involves two vertices of $L_T$, then $\pi(T)=T_{\land}=T'_{\land}=\pi(T')$. If $T$ and $T'$ differ by a rotation of two vertices that are not in $L_T$, then $\pi(T)$ and $\pi(T')$ differ by the rotation of the same pair and therefore they are adjacent in $\R(G-S)$.

Now, if $T'$ is a $W$-special search tree such that $L_T\neq L_{T'}$, or if $T'$ is not $W$-special, then $T$ and $T'$ differ by a rotation that involves $q_T$ and a vertex in $L_T\cup L_{T'}$. Hence, $\pi(T)$ and $\pi(T')$ differ by the rotation of the same pair of vertices. Thus, $\pi(T)$ and $\pi(T')$ are adjacent in $\R(G-S)$.

Finally, if $T$ and $T'$ are not $W$-special, then $\pi(T)=T$ and $\pi(T')=T'$ differ by one rotation and therefore they are adjacent in $\R(G-S)$.

Hence, $\pi$ is a quotient map.
\end{proof}

From Proposition \ref{prop_pi_quotient} follows that $\R(G-S)$ is isomorphic to the graph obtained from $\R(G)$ by contracting all the $W$-special edges.

In particular, considering the graph $G_v$, where $W=\{v,v'\}$, we define, for every $T\in V(\R(G))$,
    \[\wt P(T)=\begin{cases}
                (P(T)-\{T(d_{T,v},1),T(d_{T,v},2)\})\cup \{T_\land\}, &\text{ if $v$ is a leaf in $T$,}\\
                P(T), &\text{ otherwise.}
                \end{cases}
    \]
Additionally, we define 
    \[\wt \varepsilon_e =\begin{cases}
                (\varepsilon_e-\{T(d_{T,v},1)T'(d_{T',v},1) \mid j\in\{1,2\}\}) \cup \{T_\land T'_\land\}, &\text{ if $v$ is a leaf in $T$ and $T'$,}\\
                \varepsilon_e, &\text{ otherwise.}
                \end{cases}
    \]

Then, it becomes relatively straightforward to describe the structure of $\wt G_v$ in terms of the structure of $G_v$.

\begin{coro}
\label{coro_structure_of_ft_Gv}
    Let $G$ be a connected graph and $v\in V(G)$. 
    \begin{enumerate}[(a)]
    \item $\{\wt P(T)\mid T\in V(\R(G))\}$ is a partition of $V(\R(\wt G_v))$.
    \item For every $T\in\R(G)$, $\wt P(T)$ induces a path in $\R(\wt G_v)$. Said path is isomorphic to $\R(G_v)[P(T)]$ if $v$ is not a leaf of $T$; otherwise it is isomorphic to $(\R(G)[P(T)])/e$, where $e=T(k,1)T(k,2)$. Specifically, $\wt P(T)$ induces the path in $\wt G_v$ given by the consecutive adjacent vertices $T(0,1),\ldots,T(k-1,1),T_\land,T(k-1,2),\ldots,T(0,2)$.
    \item \label{item_ft_contractions} $\R(\wt G_v)$ can be obtained from $\R(G_v)$ by contracting edges. More precisely, if for every $T'\in V(\R(G))$ such that $v$ is a leaf of $T'$, the edge $T'(d_{T',v},1)T'(d_{T',v},2)$ is contracted in $\R(G_v)$ and the vertex resulting from the contraction is labeled $T'_\land$, we obtain $\R(\wt G_v)$.
    \item \label{item_ft_copies_of_RG} Let $j\in\{1,2\}$. Then, the subgraph of $\R(\wt G_v)$ induced by the set $\{T'(0,j)\mid T'\in V(\R(G))\}$ 
    is isomorphic to $\R(G)$.
    \item \label{item_edges_ft_Gv} $E(\R(\wt G_v)) =\left(\bigcup\limits_{T\in V(\R(G))} E(\wt P(T)) \right) \cup \left(\bigcup\limits_{e\in E(\R(G))} \wt\varepsilon_e \right)$.
    \end{enumerate}        
\end{coro}

Figure \ref{fig:RSPK22} shows the rotation graph of the complete split graph $\SPK_{2,2}$. The vertices resulting from the contraction of $\{3,4\}$-special edges in $R(K_4)$ (see Figure \ref{fig:RK4}) are the vertices $U$ and $V$.

\begin{figure}
\centering
\begin{tikzpicture}[x=0.7cm,y=0.8cm] 
    \node[vertex,label=above:{\scriptsize 4123}] at (0.5,5) (4123) {};
    \node[vertex,label=above:{\scriptsize 1423}] at (2.5,5) (1423) {};
    \node[vertex,fill=red,label=above:{\scriptsize $U$}] at (5.5,5) (1243) {};
    \node[vertex,label=above:{\scriptsize 1324}] at (8.5,5) (1324) {};
    \node[vertex,label=above:{\scriptsize 3124}] at (10.5,5) (3124) {};
    \node[vertex,label=left:{\scriptsize 4213}] at (0,4) (4213) {};
    \node[vertex,label=above:{\scriptsize 2413}] at (2,4) (2413) {};
    \node[vertex,fill=blue,label=below:{\scriptsize $V$}] at (5,4) (2143) {};
    \node[vertex,label=above:{\scriptsize 2314}] at (8,4) (2314) {};
    \node[vertex,label=right:{\scriptsize 3214}] at (10,4) (3214) {};
    \node[vertex,label=left:{\scriptsize 4132}] at (1,3) (4132) {};
    \node[vertex,label=below:{\scriptsize 1432}] at (3,3) (1432) {};
    \node[vertex,label=below:{\scriptsize 1342}] at (9,3) (1342) {};
    \node[vertex,label=right:{\scriptsize 3142}] at (11,3) (3142) {};
    \node[vertex,label=left:{\scriptsize 4231}] at (0,2) (4231) {};
    \node[vertex,label=below:{\scriptsize 2431}] at (2,2) (2431) {};
    \node[vertex,label=below:{\scriptsize 2341}] at (8,2) (2341) {};
    \node[vertex,label=right:{\scriptsize 3241}] at (10,2) (3241) {};
    \node[vertex,label=left:{\scriptsize 4312}] at (1,1) (4312) {};
    \node[vertex,label=right:{\scriptsize 3412}] at (11,1) (3412) {};
    \node[vertex,label=below:{\scriptsize 4321}] at (0.5,0) (4321) {};
    \node[vertex,label=below:{\scriptsize 3421}] at (10.5,0) (3421) {};
    \draw (4123)--(4213)--(4231)--(4321)--(4312)--(4132)--(4123);
    \draw (3124)--(3214)--(3241)--(3421)--(3412)--(3142)--(3124);
    \draw (4123)--(1423)--(1243)--(1324)--(3124);
    \draw (4213)--(2413)--(2143)--(2314)--(3214);
    \draw (4132)--(1432)--(1342)--(3142);
    \draw (4231)--(2431)--(2341)--(3241);
    \draw (4312)--(3412);
    \draw (4321)--(3421);
    \draw (1243)--(2143);
    \draw (2413)--(2431);
    \draw (2314)--(2341);
    \draw (1432)--(1423);
    \draw (1342)--(1324);
\end{tikzpicture}
    \caption{$\R(\SPK_{2,2})$}
    \label{fig:RSPK22}
\end{figure}

\section{Chromatic number of rotation graphs}
\label{section_chromatic_number}

In this section we establish results that link the application of the considered graph operations with the chromatic number of the corresponding rotation graphs. Consequently, we deduce that the chromatic number of rotation graphs of both non-complete threshold graphs and complete bipartite graphs is 3.

We start by stating a lemma characterizing permutohedra as bipartite rotation graphs, as well as rotation graphs free of 5-cycles. Although the result seems to be known, we could not find a clear reference. Thus, we include a proof for completeness.

\begin{lemma}
\label{lemma_5_cycle_bipartite_permutohedron}
Let $G$ be a connected graph with $|V(G)|\geq 3$. The following are equivalent.
\begin{enumerate}[(a)]
    \item $G$ is a complete graph.
    \item $\chi(\R(G))=2$.
    \item $\R(G)$ is free of 5-cycles.
\end{enumerate}
\end{lemma}

\begin{proof} 
    Let $n\in\N$ and $G=K_n$. Note that the search trees on $G$ are paths, and that there is a bijective correspondence between the vertices of $\R(G)$ and permutations of the set $\{1,\ldots,n\}$ (elements of the symmetric group $S_n$). For $T\in V(\R(G))$, let $\sigma_T\in S_n$ be the permutation associated with $T$. Consider the function $c\colon V(\R(G))\rightarrow \{-1,1\}$ defined by $c(T)=\text{sgn}(\sigma_T)$, where $\text{sgn}(\sigma_T)$ denotes the sign of the permutation. If $T,T'\in V(\R(G))$ differ by a rotation, then $\sigma_T$ and $\sigma_{T'}$ differ by one transposition. Therefore, $c$ is a coloring of $\R(G)$ and $\chi(\R(G))=2$, that is, $(a)$ implies $(b)$.
    
    Assume $G$ is not a complete graph. Let us prove that $\R(G)$ has a 5-cycle. Since $G$ is connected and not complete, there exists $a,c,b\in V(G)$ such that $ab,bc\in E(G)$ but $ac\notin E(G)$. Let $n=|V(G)|$ and label the vertices in $V(G)-\{a,b,c\}$ as $v_1,\ldots,v_{n-3}$. Consider the
    six vertex elimination orders 
    \begin{align*}
        &v_1,\ldots,v_{n-3},a,b,c\\
        &v_1,\ldots,v_{n-3},a,c,b\\
        &v_1,\ldots,v_{n-3},c,a,b\\
        &v_1,\ldots,v_{n-3},c,b,a\\
        &v_1,\ldots,v_{n-3},b,c,a\\
        &v_1,\ldots,v_{n-3},b,a,c.
    \end{align*}
    Since $ac\notin E(G)$, the last two vertex elimination orders determine the same search tree on $G$. Then, these are five different search trees two consecutive of them are adjacent (the adjacencies are determined by rotations of the pairs $bc,ac,ab,cb$, respectively). Also, the first and the last are adjacent since they differ by one $ab$-rotation. Therefore, the five search trees form a 5-cycle in $\R(G)$. We have proved that ($c$) implies ($a$).

    Lastly, it is clear that if $\chi(\R(G))=2$, then $\R(G)$ is bipartite and, therefore, is free of 5-cycles.
\end{proof}

Throughout the remainder of this section, we examine the conditions under which the operations in question maintain the chromatic number of the corresponding rotation graphs. We provide minimum colorings that are based on the structure of the rotation graphs described in the preceding section.

\begin{prop}
\label{prop_col_simplicial}
    Let $G$ be a connected graph that is not complete and $K\subseteq V(G)$ a subset of universal vertices in $G$. Then $\chi(\R(G_K))=\chi(\R(G))$.
\end{prop}

\begin{proof}
    By Proposition \ref{prop_copies_of_RG_in_RGK}, $\chi(\R(G_K))\geq \chi(\R(G))$.
    
    On the other hand, let $k=\chi(\R(G))$ and $c:V(\R(G))\rightarrow \{0,\ldots,k-1\}$ be a coloring of $\R(G)$. Notice that $k\geq 3$ by Lemma \ref{lemma_5_cycle_bipartite_permutohedron}. Consider $c':V(\R(G_K))\rightarrow \{0,\ldots,k-1\}$ defined as follows. If $T\in V(\R(G))$ and $i\in\{0,\ldots,\lambda(T)\}$,
    \begin{align*}
        c'(T(i)) &= \begin{cases}
            c(T), \text{ if $i$ is even,}\\
            c(T)+1 \mod k, \text{ if $i$ is odd,}
        \end{cases}\\
        c'(T(\lambda(T)+1)) &= c(T)+2 \mod k.
    \end{align*}

    Recall that from Proposition \ref{prop_edges_of_RGK}(\ref{item_edges}) the edges of $\R(G_K)$ are in a path $P(T)$ for some $T\in V(\R(G))$ or in a set $\varepsilon_e$ for some $e\in E(\R(G))$. Note also that in the last case, since every $v\in K$ is a universal vertex in $G$, there are no edges of type 2(a)ii. in $\R(G_K)$. Let $R,R'$ be adjacent vertices in $\R(G_K)$. Then, one of the following occurs.

\begin{enumerate}[(i)]
    \item there exists  $T\in V(\R(G))$ and $i\in\{0,\ldots,\lambda(T)\}$ such that $R,R'\in P^x(T)$, $R=T(i)$ and $R'=T(j)$ with $j\in\{i-1,i+1\}$, or
    \item there exist $T,T'\in V(\R(G))$ adjacent and $i\in\{0,\ldots,\lambda(T)+1\}$ such that 
    
    $R=T(i)$ and $R'=T'(i)$, or $R=T(\lambda(T)+1)$ and $R'=T'(\lambda(T')+1)$.
\end{enumerate}
   
    Thus, if case (i) occurs, then $c'(R)\neq c'(R')$ since $k\geq 3$. If case (ii) occurs, then $c'(R)\neq c'(R')$ since $c(T)\neq c(T')$.

    Therefore, $c'$ is a coloring of $\R(G_K)$, which implies $\chi(\R(G_K))\leq \chi(\R(G))$.
\end{proof}

Recall that if $K = \{v\}$ for some $v\in V(G)$, then $x$ is a pendant vertex adjacent to $v$ in $G_K$ and that for every $T\in V(\R(G))$, $v_{\lambda(T)}=v$. Thus, as a corollary to the previous proposition, we derive that adding a pendant vertex to a universal vertex in a non-complete graph preserves the chromatic number of the corresponding associahedron.

\begin{coro}
\label{coro_chi_pendant}
    Let $G$ be a connected graph that is not complete and $v\in V(G)$ be a universal vertex. Then $\chi(\R(G))=\chi(\R(G_{\{v\}}))$.
\end{coro}

In the subsequent result, it is proved that if $G$ is not complete and contains a universal vertex, the chromatic number of the rotation graph remains unchanged upon the addition of another universal vertex to $G$.

\begin{prop}
\label{prop_chi_ttwin}
Let $G$ be a connected graph and let $v\in V(G)$ be a universal vertex in $G$. Then $\chi(\R(G_v)) = \chi(\R(G))$.
\end{prop}

\begin{proof}
By Proposition \ref{prop_copies_of_AG}, $\chi(\R(G_v))\geq \chi(\R(G))$.

In order to prove that $\chi(\R(G_v))\leq \chi(\R(G))$, we give a coloring of $\R(G_v)$ using $\chi(\R(G))$ colors. Let $k=\chi(\R(G))-1$ and $c:V(\R(G))\rightarrow \{0,\ldots,k\}$ be a coloring of $\R(G)$. Recall from Lemma \ref{lemma_5_cycle_bipartite_permutohedron} that $k\geq 2$. Consider $c':V(\R(G_v))\rightarrow \{0,\ldots,k\}$ defined as follows. For every $T\in V(\R(G))$ and $i\in\{0,\ldots,d_{T,v}\}$,
\[c'(T(i,j))=\begin{cases}
    c(T), &\text{ if } i \text{ is even and } j=1,\\
    c(T)+1 \mod k+1, &\text{ if } i \text{ is odd and } j=1,\\
    c(T), &\text{ if } i \text{ is odd and } j=2,\\
    c(T)+1 \mod k+1, &\text{ if } i \text{ is even and } j=2.
\end{cases}
\]

Recall that from Proposition \ref{prop_edges_of_RGv} the edges of $\R(G_v)$ are in a path $P(T)$ for some $T\in V(\R(G))$ or in a set $\varepsilon_e$ for some $e\in E(\R(G))$. Note also that in the last case, since $v$ is a universal vertex in $G$, there are no edges of type 2(a)ii. in $\R(G_v)$. Let $R,R'$ be adjacent vertices in $\R(G_v)$. Then, one of the following occurs.
\begin{enumerate}[(i)]
    \item There exist $T\in V(\R(G))$, $i\in\{0,\ldots,d_{T,v})\}$ and $j\in\{1,2\}$ such that $R,R'\in P(T)$; moreover $R=T(i,j)$ and $R'=T(i',j)$ with $i'\in\{i-i,i+1\}$, or $R=T(d_{T,v},1)$ and $R'=T(d_{T,v},2)$.
    \item There exist $T,T'\in V(\R(G))$ adjacent in $\R(G)$, $i\in\{0,\ldots,d_{T,v}\}$, and $j\in\{1,2\}$ such that $R=T(i,j)$ and $R'=T'(i,j)$.
\end{enumerate}

In case (i), $c'(R)\neq c'(R')$ since $k\geq 3$, and in case (ii), $c'(R)-c'(R')=c(T)-c(T')\neq 0$ since $c$ is a coloring. Therefore $c'$ is a coloring of $\R(G_K)$. 
\end{proof}

It is known that a connected threshold graph $G$ can be constructed from a single vertex by a sequence of operations, each of which consists of adding either an isolated vertex or a universal vertex. Furthermore, if $G$ is connected, this process ends with the addition of a universal vertex. 

However, adding isolated $k$ vertices and then adding a universal vertex $u$ is equivalent to first adding $u$ as a universal vertex and after that, $k$ pendant vertices adjacent to $u$. Notice also that, after adding those $k$ isolated vertices, $u$ remains a universal vertex. Therefore, the subsequent addition of a new universal vertex is equivalent to introducing a true twin to a universal vertex. Then, $G$ can be obtained by successively applying the operations of adding either a true twin to a universal vertex or a pendant vertex adjacent to a universal vertex. By Corollary \ref{coro_chi_pendant} and Proposition \ref{prop_chi_ttwin} these operations preserve the chromatic number of the corresponding graph associahedra. Since the path graph on three vertices is a connected threshold graph and the chromatic number of its rotation graph (the 5-cycle) is 3, we derive the following corollary.

\begin{coro}
\label{coro_thresholds}
Let $G$ be a non-complete connected threshold graph. Then $\chi(\R(G))=3$. 
\end{coro}

\begin{obs}
    Since split complete graphs (including star graphs) are threshold graphs, Corollary \ref{coro_thresholds} determines that the chromatic number of their rotation graphs is 3.
\end{obs}

Finally, we prove that adding a false twin to a universal vertex in a graph preserves the chromatic number of the associated rotation graphs. In fact, a more general result holds, as we state in the following proposition.

\begin{prop} 
    Let $G$ be a non-complete graph and $V_1,V_2$ a partition of $V(G)$ such that $N_G(u)=V_2$ for all $u\in V_1$. Let $v\in V_1$. Then $\chi(\R(\wt G_v)))=\chi(\R(G))$.
\end{prop}

\begin{proof}
    By Corollary \ref{coro_structure_of_ft_Gv}(\ref{item_ft_copies_of_RG}), $\chi(\R(\wt G_v))\geq \chi(\R(G))$.
   
    Let $k=\chi(\R(G))\geq3$ and $c:V(\R(G))\rightarrow\{0,\ldots,k-1\}$ be a coloring of $G$. Consider $c':V(\R(\wt{G}_v))\rightarrow\{0,\ldots,k-1\}$ defined as follows. 
    For $i\in\{0,\ldots,d_{T,v}\}$ and $j\in\{1,2\}$,
    \begin{align*}
    c'(T(i,j))&=\begin{cases}
        c(T), &\text{if } i \text{ is even and } j=1 \text{; or if } i \text{ is odd and } j=2,\\
        c(T)+1 \mod k, &\text{if } i \text{ is odd and } j=1 \text{; or if } i \text{ is even and } j=2,
    \end{cases}\\
    c'(T_\land)&=c(T)+2 \mod k. 
    \end{align*}
    Let us see that $c'$ is a coloring of $\R(\wt{G}_v)$. 

    Recall that from Corollary \ref{coro_structure_of_ft_Gv}, the edges of $\R(\wt G_v)$ are in a path $\wt P(T)$ for some $T\in V(\R(G))$ or in a set $\wt\varepsilon_e$ for some $e\in E(\R(G))$. Let $R,R'$ be adjacent vertices in $\R(\wt G_v)$. Then, one of the following occurs.
    \begin{enumerate}[(i)]
        \item There exist $T\in V(\R(G))$, $i\in\{0,\ldots,d_{T,v}\}$ and $j\in\{1,2\}$ such that $R=T(i,j)$, $R'=T(i',j)$ for some $i'\in\{i-1,i+1\}$, or $R=T(d_{T,v},1)$ and $R'=T(d_{T,v},2)$ (or vice versa).
        \item There exist $T\in V(\R(G))$ and $j\in\{1,2\}$ such that $R=T(d_{T,v}-1,j)$ and $R'=T_\land$ (or vice versa).
        \item There exist $T,T'\in V(\R(G))$ adjacent such that $R=T(i,j)$ and $R'=T'(i',j)$ for some $i'\in\{i-1,i,i+1\}$, or $R=T_\land$ and $R'=T'_\land$.
    \end{enumerate} 

    If the case (i) or (ii) occurs, then $c'(R)\neq c'(R')$ since $k\geq 3$. If case (iii) occurs with $R=T(i,j)$ and $R'=T'(i,j)$ or $R=T_\land$ and $R'=T'_\land$, then $c'(R)\neq c'(R')$ since $c(T)\neq c(T')$. 

    Let us see that case (iii) for $R=T(i,j)$ and $R'=T'(i',j)$ with $i'\neq i$ never occurs. According to
    Corollary \ref{coro_structure_of_ft_Gv} and the description of edge types from Section \ref{subsection_adding_tt}, $i\neq i'$ only if $T$ differs from $T'$ by an $ab$-rotation, where $a,b\in V(T)$ are in the path $A$ in $T$ connecting $r_T$ with $v$. W.l.o.g. assume that $a$ is the parent of $b$ in $T$. Then $b\neq v$ and no vertex of the subtree $S$ of $b$ in $T$ that contains $v$ is adjacent to $a$ in $G$. In view of the facts that $N_G(u)=V_2$ for all $u\in V_1$ and $v\in V_1$, then $a\in V_1$, $S\subseteq V_1$ and $b\in V_2$. Since $V_1$ is an independent set of vertices in $G$, $V(S)=\{v\}$. Thus, $v$ is a leaf in $T$, while in $T'$, both $a$ and $v$ are leaves. It follows that $R=T(i,j)$ and $R'=T'(i,j)$ for some $i\in\{0,\ldots,d_{T,v}-2\}$ and $j\in\{1,2\}$, or $R=T_\land$ and $R'=T'_\land$ (notice also that $T(d_{T,v}-1,j)$ is not adjacent to any vertex in $\wt P(T')$).
    
    Thus, $c'$ is a coloring of $\R(\wt G_v)$ and therefore $\chi(\R(\wt G_v))\leq \chi(\R(G))$.
\end{proof}

\begin{coro}
\label{coro_chi_ft}
    Let $G$ be a non-complete connected graph and $v\in V(G)$ a universal vertex. Then $\chi(\R(\wt G_v))=\chi(\R(G))$.
\end{coro}

\begin{obs}
    Since $\chi(K_{1,2})=3$, Corollary \ref{coro_chi_ft} implies that the chromatic number of the rotation graph of complete bipartite graphs $K_{p,q}$ is $3$ (provided $p\geq 2$ or $q\geq 2$).
\end{obs}

\section{On distances and diameter of rotation graphs}
\label{section_dist_and_diam}

The problem of finding distances in rotation graphs, as well as determining or bounding their diameter, has attracted attention in recent research. To address this, we begin by presenting a result that connects the adjacency of vertex pairs in a graph $G$ with the number of rotations they undergo in sequences of rotations in 
$\R(G)$. This result will be useful for the analysis of distances and diameters in rotation graphs that follows.

If $T,T'\in V(\R(G))$ and $u,v\in V(G)$ are vertices such that $u$ is an ancestor of $v$ in $T$ and a descendant of $v$ in $T'$, then we say that $u,v$ have different relative order in $T$ and $T'$. If $u$ is an ancestor (alternatively descendant) of $v$ in both $T$ and $T'$, then we say that $u,v$ have the same relative order in $T$ and $T'$.

\begin{lemma}
\label{lemma_adj_vertices_rotation_occurs_once}
Let $G$ be a graph and let $u,v \in V(G)$.
\begin{enumerate}[(a)]
    \item \label{item_relative_order} Suppose $u$ and $v$ are adjacent in $G$ and let $T,T'\in V(\R(G))$. Then, $u$ and $v$ have different relative order in $T$ and $T'$ if and only if any sequence of rotations that transforms $T$ into $T'$ has an odd number of $uv$-rotations.
    \item \label{item_at_most_one_v1v2_rotation} If $u$ and $v$ are true twins 
    in $G$ and $T,T'\in V(\R(G))$, then every minimum length $TT'$-path in $\R(G)$ has exactly one edge determined by an $uv$-rotation if $u$ and $v$ have different relative orders in $T$ and in $T'$, and no edge determined by $uv$-rotations if $u$ and $v$ have the same relative order in $T$ and in $T'$. 
\end{enumerate}
\end{lemma}

\begin{proof}\ 

($a$) The result is a direct consequence of Remark \ref{obs_uv_branch}.

($b$) Let $T,T'\in  V(\R(G))$ and $A$ a minimum lenght $TT'$-path in $\R(G)$. Suppose $A$ has vertices $T=T_0,\ldots,T_k=T'$ and that $T_iT_{i+1}$ are adjacent for $i\in\{0,\ldots,k-1\}$.

Let $u,v\in V(G)$ be true twins in $G$ and assume two or more edges in $A$ correspond to $uv$-rotations. Let $i,j\in\{0,\ldots,n-1\}$ be indices such that $i+1<j$, and that $T_i,T_{i+1}$ and $T_jT_{j+1}$ are edges determined by two of these rotations. 

Consider $\rho: V(G)\rightarrow V(G)$ the function defined by $\rho(u)=v$, $\rho(v)=u$ and $\rho(x)=x$ for every $x\in V(G)-\{u,v\}$. Notice that since $u$ and $v$ are true twins in $G$, then $\rho$ is an isomorphism. Additionally, for every $T\in V(\R(G))$, we denote by $\rho^{\ast}(T)$ the rooted tree with vertex set $V(G)$ and edge set $\{\rho(x)\rho(y) \mid xy\in E(T)\}$. Since $\rho$ is a graph isomorphism, $\rho^{\ast}(T)$ is a search tree on $G$. Further notice that $T_i=\rho^{\ast}(T_{i+1})$ and $T_{j+1}=\rho^{\ast}(T_{j})$.

We have that the search trees
\[T=T_0,\ldots,T_{i-1},\rho^\ast(T_{i+1}),\ldots, \rho^\ast(T_{j}),T_{j+2},\ldots,T_k=T'
\]
form a $TT'$-walk with exactly two edges (determined by $uv$-rotations) less than $A$. This is absurd since $A$ has minimum length.

Therefore, $A$ has at most one edge corresponding to a $uv$-rotation. By item (\ref{item_relative_order}), there is exactly one $uv$-rotation if $u$ and $v$ have different relative order in $T$ and $T'$, and zero otherwise.
\end{proof}

\begin{obs}
\label{obs:selectRotation}
If $u$ and $v$ are true twins of $G$, the proof of item $(b)$ gives a procedure applicable to a $TT'$-path $A$ in $\R(G)$, that allow us to choose two edges determined by $uv$-rotations and to construct a new $TT'$-walk without those particular edges.
\end{obs}

As a direct consequence of Lemma \ref{lemma_adj_vertices_rotation_occurs_once}$(b)$, we obtain the following.

\begin{prop}
    Let $G$ be a connected graph, $v\in V(G)$, $T,T'\in V(\R(G_v))$. 
    \begin{enumerate}[(a)]
        \item Let $j\in\{1,2\}$. If $S,S'\in V(\A_j)$ then the paths of minimum length joining $S$ and $S'$ are all contained in $\A_j$.
        \item If $S\in V(\A_1)$ and $S'\in V(\A_2)$, then the paths of minimum length between $S$ and $S'$ have exactly one edge determined by a $vv'$-rotation.
    \end{enumerate}
\end{prop}


Recall from Proposition \ref{prop_pi_quotient} that we can describe the structure of $\R(G-S)$ in terms of the structure of $\R(G)$ as a graph quotient. This implies that there is a relation between distances in both rotation graphs. In the remainder of the section, we utilize this relation to establish a lower bound for $\diam (\R(G-S))$ in terms of $\diam(\R(G))$. 


In the following remark we state a particular property of the map $\pi$. 

\begin{obs}
\label{obs_preimage_by_pi_is_permutohedron}
Let $G$ be a connected graph, $W\subseteq V(G)$ a set of true twins and $S=E(G[W])$. Let $\wt T\in V(\R(G-S))$. If there exists $L\subseteq W$ such that $|L|\geq 2$ and every $l\in L $ is a leaf of $\wt T$ with a common parent $q\in V(G)-W$, then the subgraph of $\R(G)$ induced by $\pi^{-1}(\wt T)$ is isomorphic to $\R(K_{|L|})$. Otherwise, $\pi^{-1}(\wt T)$ has a single element.
\end{obs}

Notice that $\R(G-S)$ is isomorphic to the graph obtained from $\R(G)$ after applying several edge contractions of $W$-special edges. In addition, if $T,T'\in V(\R(G))$, from Lemma \ref{lemma_adj_vertices_rotation_occurs_once}(\ref{item_at_most_one_v1v2_rotation}) we known that in any minimum $TT'$-path there are at most $\binom{|W|}{2}$ $W$-special edges. Combining this fact and previous remark we can set the following result. 

\begin{lemma}
\label{lemma_paths_gral}
Let $G$ be a connected graph, $W\subseteq V(G)$ a set of true twins such that $|W|\geq 2$ and $S=E(G[W])$. Let $T,T'\in V(\R(G))$ and $\widetilde A$ a $\pi(T)\pi(T')$-path in $\R(G-S)$. Then, there exists a $TT'$-path $A$ in $\R(G)$  such that
\[
\len(A)\leq \len(\widetilde A)+\binom{|W|}{2}.
\]
\end{lemma}

\begin{proof}
Assume first that $\pi(T)\neq \pi(T')$ and let $\widetilde A$ be a $\pi(T)\pi(T')$-path in $\R(G-S)$. Suppose that $\wt A$ consists in a sequence of vertices $\pi(T)=\widetilde T_0,\widetilde T_1, \ldots,\widetilde T_n=\pi(T')$ for some $n\geq 1$ and that for $i\in\{0, \ldots, n-1\}$, $\widetilde T_{i+1}$ is obtained from  $\widetilde T_i$ by a $u_iv_i$-rotation, for some $u_i,v_i\in V(G)$ adjacent vertices in $\wt T_i$.

For every $i\in\{0,\ldots,n-1\}$, let $T_i\in \pi^{-1}(\widetilde T_i)$ such that $u_i$ and $v_i$ are adjacent in $T_i$. By making an $u_iv_i$-rotation on $T_i$ we obtain $T'_{i+1}\in\pi^{-1}(\wt T_{i+1})$. Let $B_{i+1}$ a (possibly trivial) $T'_{i+1}T_{i+1}$-path in the subgraph of $\R(G)$ induced by $\pi^{-1}(\wt T_{i+1})$ (recall that such a path exists by Remark \ref{obs_preimage_by_pi_is_permutohedron}).

The concatenation of the paths $B_i$ (whose edges are $W$-special edges) and the $u_iv_i$-rotations for $i\in\{0,\ldots, n-1\}$ give a $TT'$-path $B$ in $\R(G)$. Thus, by Remark \ref{obs:selectRotation}, we can construct a path $A$ from $B$ by applying the process in item $(b)$ of Lemma \ref{lemma_adj_vertices_rotation_occurs_once} between pairs of $uv$-rotations ($u,v\in W$) that determines $W$-special edges. Moreover, the path $A$ has at most one more rotation for each pair of vertices in $W$ (that correspond to $W$-special edges), and therefore $A$ has length at most $n+\binom{|W|}{2}=\len(\widetilde A)+\binom{|W|}{2}$.

If $\pi(T)=\pi(T')$, then $\wt A$ is a path of lenght $0$, and by Remark \ref{obs_preimage_by_pi_is_permutohedron} there exists a $TT'$-path in $A$ of lenght at most $\binom{|W|}{2}=\diam(\R(K_{|W|}))$.
\end{proof}

\begin{obs}
\label{obs_W_special_edges_in_path}
In the previous proposition, the constructed path $A$ has the same number of rotations that $\widetilde A$ plus at most one rotation between each pair of vertices in $W$ (that corresponds to a $W$-special edge). The length of the path $A$ is increased with respect to the length of $\wt A$ by rotations of pairs of vertices from $W$.
\end{obs}

\begin{theo}
\label{theo_lower_bound_quotient}
Let $G$ be a connected graph and let $W\subseteq V(G)$ be a set of true twins in $G$ such that $|W|\geq 2$. Let $S=E(G[W])$. Then
\[\diam(\R(G)) - \binom{|W|}{2}\leq \diam(\R(G-S)).
\]
\end{theo}

\begin{proof}
Let $T, T'\in V(\R(G))$. Let $\widetilde B$ a $\pi(T)\pi(T')$-path in $\R(G-S)$ of minimum length. By Lemma \ref{lemma_paths_gral}, there exists $B$ a $TT'$-path in $\R(G)$ such that
    \[
    \len (B)\leq \len(\widetilde B)+\binom{|W|}{2}.
    \]
Let $A$ be a $TT'$-path in $\R(G)$ of minimum length.  Then
     \begin{align*}
    \dist_{\R(G)}(T,T')&=\len(A)\\
    &\leq \len(B)\\
    &\leq \len(\widetilde B)+\binom{|W|}{2}\\
    &=\dist_{\R(G-S)}(\pi(T),\pi(T'))+\binom{|W|}{2}\\
    &\leq \diam(\R(G-S))+\binom{|W|}{2},
    \end{align*}
and therefore
    \[
    \diam(\R(G))\leq \diam(\R(G-S))+\binom{|W|}{2}.
    \]
\end{proof}

The bound in Theorem \ref{theo_lower_bound_quotient} turns out to be sharp for the case when $G=\SPK_{p,q}$ with $q\geq 4p+1$, and $W=P$. Indeed, for $p,q\in\N$, the complete bipartite graph $K_{p,q}$ is $\SPK_{p,q}-S$ where $S$ is the set of edges with both endpoints in the clique $P$. Since $P$ is a set of true twins in $\SPK_{p,q}$, by Theorem
\ref{theo_lower_bound_quotient},
    \begin{equation}
    \label{eq_new_bound}
    \diam(\R(\SPK_{p,q}))-\binom{p}{2}\leq \diam(\R(K_{p,q})).
    \end{equation}
    Cardinal et al. show in \cite{CPV-2022a} that
\[
\diam(\R(\text{SPK}_{p,q}))=\begin{dcases}
2pq+\binom{p}{2}=2m-\binom{p}{2}, \text{ if }q\geq 4p+1\\
pq+\left\lfloor\frac{1}{2}\binom{q}{2}\right\rfloor+\binom{p}{2}, \text{ otherwise.}
\end{dcases}
\] 
They also show that $\diam(\R(K_{p,q}))=2pq$ for unbalanced complete bipartite graphs, that is, for $q\geq 4p+1$, and give the following bounds for the balanced case, when $\frac{p}{4}\leq q\leq 4p$. 
\begin{equation}
\label{eq_bound_binom_q2}
\binom{q}{2}\leq \diam(\R(K_{p,q}))\leq 2pq.
\end{equation}
Thus, as we claimed, when $q\geq 4p+1$ the bound in Theorem \ref{theo_lower_bound_quotient} is sharp for $\R(\text{SPK}_{p,q})$ and $\R(\text{K}_{p,q})$. Furthermore, when $q\leq 4p$, we obtain a new lower bound for the diameter of associahedra of balanced complete bipartite graphs, which improves on the one presented in \cite{CPV-2022a}.
    




\begin{coro}
\label{coro_lb_Kpq}
If $\min\{2,\frac{p}{4}\}\leq q\leq 4p$, 
\[
pq+\left\lfloor\frac{1}{2}\binom{q}{2}\right\rfloor\leq \diam(\R(K_{p,q})).
\]
\end{coro}




In the following, we apply this result to the computation of the exact value of $\diam(K_{2,q})$, for balanced graphs $K_{2,q}$, that is, for $3\leq q \leq 8$. In that case we have that

\begin{equation}
\label{eq_diam_SPK-1}
 \diam(\R(\text{SPK}_{2,q}))-1\leq \diam(\R(K_{2,q}))\leq \diam(\R(\text{SPK}_{2,q})),
\end{equation}
where the first inequality follows from Theorem \ref{theo_lower_bound_quotient} and the second holds since $K_{2,q}$ is a subgraph of $\SPK_{2,q}$ \cite{MP-2015}.

To achieve this, we first establish a result analogous to Lemma \ref{lemma_paths_gral}. Furthermore, for this case a stronger result can be obtained, as we explicitly determine distances in $\R(\wt G_v)$ in terms of distances in $\R(G)$. 

\begin{lemma}
\label{lemma_distance}
Let $G$ be a graph and $W=\{u,v\}\in V(G)$ such that $u$ and $v$ are true twins. Let $T,T'\in V(\R(G))$. 
\begin{enumerate}[(a)]
    \item If $\wt A$ is a $\pi(T)\pi(T')$-path in $\R(G-uv)$, there exists a $TT'$-path $A$ in $\R(G)$ such that $\len(\widetilde A)\leq \len(A)\leq \len(\widetilde A)+1$. Additionally, if $\len(A)=\len(\widetilde A)$, then no edge of $A$ is  a $W$-special edge, and if $\len(A)=\len(\widetilde A)+1$, then $A$ has exactly one $W$-special edge.
    \item If there exists a $TT'$-path of minimum length that has a $W$-special edge, then 
    \[
    \dist_{\R(G-uv)}(\pi(T),\pi(T')) = \dist_{\R(G)} (T,T')-1.
    \] 
    Otherwise,  $\dist_{\R(G-uv)}(\pi(T),\pi(T'))=\dist_{\R(G)} (T,T')$.
\end{enumerate}
\end{lemma}

\begin{proof}\ 

($a$) It is a direct consequence of Lemma \ref{lemma_paths_gral} and Remark \ref{obs_W_special_edges_in_path}.

($b$) Let $T,T'\in V(\R(G))$ and let $A$ be a minimum length $TT'$-path. If $A$ has consecutive vertices $T_0,T_1,\ldots,T_k$, we denote $\pi(A)$ as the path in $\R(G-S)$ with vertex set $\{\pi(T_i)\mid i\in\{0,\ldots,k\}\}$ and with edges $\pi(T_i)\pi(T_{i+1})$ for $i\in\{0,\ldots, k-1\}$ such that $\pi(T_i)\neq\pi(T_{i+1})$.

Suppose that $A$ has (exactly) one $W$-special edge. Then $\len(\pi(A))=\len(A)-1$. Assume that there exists a shorter $\pi(T)\pi(T')$-path $\widetilde A_0$. We construct a path $A_0$ in $\R(G)$, as in item ($a$), such that $\len(A_0)\leq\len(\widetilde A_0)+1$. Hence, $\len(A_0)\leq \len(\widetilde A_0)+1< \len(\pi(A))+1=\len(A)$. But this is a contradiction, since $A$ had minimum length. Therefore, $\dist_{\R(G-uv)}(\pi(T),\pi(T'))=\len(\pi(A))=\len(A)-1=\dist_{\R(G)} (T,T')-1$.

Now suppose that no $TT'$-path of minimum length has a $W$-special edge. We have that $\pi(A)$ is a $\pi(T)\pi(T')$-path in $\R(G-uv)$, and that, by definition of $\pi$, $\len(\pi(A))=\len(A)$. Suppose there exists a shorter $\pi(T)\pi(T')$-path $\widetilde A_0$. We construct a path $A_0$ in $\R(G)$, as in item ($a$), such that $\len(A_0)\leq \len(\widetilde A_0)+1$. Then $\len(A_0)\leq \len(\widetilde A_0)+1< \len(\pi(A))+1=\len(A)+1$. Hence, $\len(A_0)\leq \len(A)$, but since $A$ is of minimum length, $\len(A_0)=\len(A)$. Then, by the supposition at the beginning of the paragraph, $A_0$ has no $W$-special edge and thus $\len(A_0)=\len(\widetilde A_0)$ by item ($a$), with $\len(\widetilde A_0)<\len(\pi(A))=\len(A)$. This is a contradiction, since $A$ had minimum length. Therefore, $\dist_{\R(G-uv)}(\pi(T),\pi(T'))=\len(\pi(A))=\len(A)=\dist_{\R(G)} (T,T')$.
\end{proof}

\subsection*{Computation of diameters of associahedra of complete bipartite graphs}

From Equation \ref{eq_diam_SPK-1} we have that 
\begin{align*}
7&\leq \diam(\R(K_{2,3}))\leq 8\\
11&\leq \diam(\R(K_{2,4}))\leq 12\\
15&\leq \diam(\R(K_{2,5}))\leq 16\\
19&\leq \diam(\R(K_{2,6}))\leq 20\\
24&\leq \diam(\R(K_{2,7}))\leq 25\\
30&\leq \diam(\R(K_{2,8}))\leq 31.
\end{align*}

In this subsection, we demonstrate that 
\begin{align*}
\diam(\R(K_{2,3}))&= 8\\
\diam(\R(K_{2,4}))&= 11\\
\diam(\R(K_{2,5}))&= 15\\
\diam(\R(K_{2,6}))&= 20\\
\diam(\R(K_{2,7}))&= 25\\
\diam(\R(K_{2,8}))&= 30.
\end{align*}

While computational methods could be employed to obtain these values, we present a formal proof to enhance theoretical comprehension and ensure rigorous accuracy.
 
Search trees on complete split graphs are called brooms, since they consist in a path such that one of its endpoints is the root and the other is attached to the leaves \cite{CPV-2022a}. Said path is called the handle of the broom. Figure \ref{fig:exampleSpecial} shows three brooms on $\SPK_{3,3}$. 

First, we make some general considerations following \cite{CPV-2022b}. Let $T,T'\in V(\R(K_{2,q}))$. For every $i=1,\ldots,q$, let $w_i$ be the number of vertices of $P$ that are above of $y_i$ in $T$ and let $w'_i$ be the number of vertices of $P$ that are above of $y_i$ in $T'$. We consider two $TT'$-paths in $\R(K_{2,q})$. The first of them consists in making every vertex of $Q$ in the handle of $T$ a leaf, using exactly
\[
\sum_{i=1}^q (2-w_i) = 2q-\sum_{i=1}^q w_i
\]
rotations. Then rearrange the two vertices of $P$ in the handle if necessary, using at most one rotation. Finally, take the vertices of $Q$ to their place in $T'$ using
\[
\sum_{i=1}^q (2-w'_i) = 
2q-\sum_{i=1}^q w'_i
\]
rotations. This path has length at least 
\[
1+4q-\sum_{i=1}^q w_i+w'_i.
\]

The other path consists in making the vertices of $P$ leaves in $T$, using 
\[
\sum_{i=1}^q w_i
\]
rotations. Analogously, it is possible to transform $T'$ in a tree with the vertices of $P$ as leaves, without making rotations between pairs of vertices of $Q$, using
\[
\sum_{i=1}^q w'_i 
\]
rotations. Transforming one of these trees with the vertices of $Q$ in the handle into the other, can be done using at most $\binom{q}{2}$ rotations. Hence, this $TT'$-path has length at most
\[
\binom{q}{2}+\sum_{i=1}^q w_i+w'_i .
\]
Therefore, $\dist_{\R(K_{2,q})}(T,T')\leq\min\left\{1+4q-\Omega, \binom{q}{2}+\Omega\right\}$, where $\Omega=\sum_{i=1}^q w_i+w'_i$. This bound is maximized when $\Omega=2q-\frac{1}{2}\left[\binom{q}{2}-1\right]$, thus 
\[
\dist_{\R(K_{2,q})}(T,T')\leq 2q+\frac{1}{2}\left[\binom{q}{2}+1\right].
\]

Now, for $q=4,5,8$ this upper bound is 11.5, 15.5, and 30.5, respectively, yielding 
\begin{align*}
\diam(\R(K_{2,4}))&\leq 11\\
\diam(\R(K_{2,5}))&\leq 15\\
\diam(\R(K_{2,8}))&\leq 30.
\end{align*}

As in all these cases the upper bound coincides with the lower bound given in (\ref{eq_diam_SPK-1}), the equalities hold. For the remaining values of $q$, there exists a gap of 1 between the lower and the upper bound.  Therefore, a different strategy is necessary to establish the exact value of the diameter.


In order to compute the diameter of $\R(K_{2,q})$ for $q=3,6,7$, in every case we follow a similar strategy. We give a pair of trees $T,T'\in V(\R(\SPK_{2,q}))$ such that $\dist_{\R(\SPK_{2,q})}(T,T')=\diam(\R(\SPK_{2,q}))$ and we prove that no path of minimum length between them has a $x_1x_2$-rotation (the true twin vertices in the graph). We then apply Lemma \ref{lemma_distance}, to obtain that $\dist_{\R(K_{2,q})}(\pi(T),\pi(T'))=\dist_{\R(\SPK_{2,q})}(T,T')=\diam(\R(\SPK_{2,q}))$. This gives the lower bound $\diam(\R(\SPK_{2,q}))\leq \diam(\R(K_{2,q}))$. We also have, since $K_{2,q}\subseteq \SPK_{2,q}$, the upper bound $\diam(\R(K_{2,q}))\leq \diam(\R(\SPK_{2,q}))$. The equalities follow. 
\medskip

We prove first that $\diam(\R(K_{2,3}))=8$. Recall that $\dist_{\R(\SPK_{2,3})}(T,T')\leq \diam(\SPK_{2,3})=8$. Let us see that for the search trees $T,T'$ on $\SPK_{2,3}$ shown in Figure \ref{fig:searchTreesOnSPK23}, $\dist_{\R(\SPK_{2,3})}(T,T')=8$. 

\begin{figure}
\centering
\begin{tikzpicture}
\node[vertex,label=right:{\small $y_1$}] at (0,3) (y1) {};
\node[vertex,label=right:{\small $x_1$}] at (0,2.5) (x1) {};
\node[vertex,label=right:{\small $y_2$}] at (0,2) (y2) {};
\node[vertex,label=right:{\small $x_2$}] at (0,1.5) (x2) {};
\node[vertex,label=right:{\small $y_3$}] at (0,1) (y3) {};
\draw (y1)--(x1)--(y2)--(x2)--(y3);
\node at (0,0.3) (T) {$T$};
\end{tikzpicture}
\qquad
\begin{tikzpicture}
\node[vertex,label=right:{\small $y_3$}] at (0,3) (y3) {};
\node[vertex,label=right:{\small $y_2$}] at (0,2.5) (y2) {};
\node[vertex,label=right:{\small $x_2$}] at (0,2) (x2) {};
\node[vertex,label=right:{\small $x_1$}] at (0,1.5) (x1) {};
\node[vertex,label=right:{\small $y_1$}] at (0,1) (y1) {};
\draw (y1)--(x1)--(y2)--(x2)--(y3);
\node at (0,0.3) (T') {$T'$};
\end{tikzpicture}

    \caption{Two search trees on $\SPK_{2,3}$.}
    \label{fig:searchTreesOnSPK23}
\end{figure}
 
To show that $\dist_{\R(\SPK_{2,3})}(T,T')\geq 8$, consider a sequence of rotations that transforms $T$ into $T'$. By Proposition (\ref{item_relative_order}), in this sequence there must be at least one $x_1x_2$-rotation, one $y_1x_i$-rotation and one $y_3x_i$-rotation  for $i=1,2$, since these pairs of vertices are adjacent and have different relative order in $T$ and $T'$. Then, this sequence has at least 5 rotations.

On the other hand, it can be the case that $y_2$ is a leaf of some broom in the sequence or not. In the first case, there must be a rotation with $x_2$ (going down), and then rotations with $x_1$ and $x_2$ (going back up) $T'$. In the second case, there are rotations with $x_1$, $y_1$ and $y_3$ since they have different order in $T$ and $T'$. In both cases, at least three more rotations are needed.

Thus, if $A$ is a $TT'$-path, then $\len(A)\geq 8$, which implies $\dist_{\R(\SPK_{2,3})}(T,T')\geq 8$. Therefore $\dist_{\R(\SPK_{2,3})}(T,T')= 8$. Denoting by $k$ the number of vertices of $Q$ that are leaves in any of the brooms in the sequence, these 8 necessary rotations are listed in Table \ref{table_rotations_3}.

\begin{table}
\centering
\begin{tabular}{ |c|c|c| }
\hline
Vertices involved & Effect of rotation & No. of rotations\\
\hline
$x_1,x_2$ & exchange & 1\\
\hline
$y_1$ with $x_1,x_2$ & $y_1$ leaf goes down & 2\\
\hline
$y_3$ with $x_1,x_2$ & $y_3$ leaf goes up & 2\\
\hline
\multicolumn{3}{|c|}{Case $k=2$} \\
\hline
$y_2$ with $y_1,y_3$ & $y_2$ exchange within handle & 2 \\ 
\hline
$x_1,y_2$ & exchange within handle & 1\\
\hline
\multicolumn{3}{|c|}{Case $k=3$} \\
\hline
$y_2,x_2$ & $y_2$ leaf goes down  & 1 \\
\hline
$y_2$ with $x_1,x_2$ & $y_2$ leaf goes up & 2\\
\hline
\end{tabular}
\caption{Necessary rotations in a sequence that transforms $T$ into $T'$.}
\label{table_rotations_3}
\end{table}

We show now that no $TT'$-path of minimum length in $\R(\SPK_{2,3})$ has a $x_1x_2$-rotation. This, together with Lemma \ref{lemma_distance} allows to conclude that $\dist_{\R(K_{2,3})}(T,T')=8$.

Suppose there were a minimum length path containing a $P$-special edge.

Note that $y_3$ is a leaf in the first tree of the sequence, before the occurrence of this $P$-special edge, and $y_1$ is a leaf in a tree that is after the special edge in the sequence. If $y_1$ and $y_3$ are not simultaneously leaves in any tree of the sequence, then there must be a rotation $y_1y_3$ (in the handle of a broom), not counted in Table \ref{table_rotations_3}. Then this path would not be optimum. It follows that $y_1$ and $y_3$ are simultaneously leaves in some tree of the path.

But, if $y_3$ is again a leaf after the special edge, that would require two additional rotations (between the pairs $x_1y_3$, $x_2y_3$), not counted in Table \ref{table_rotations_3}. Analogously, if  $y_1$ were a leaf in some tree of the sequence before the special edge, that would require two additional rotations. Thus, such a sequence would not be optimal.

Therefore no $TT'$-path in $\R(\SPK_{2,3})$ has a $P$-special edge.

It follows that $\diam(\R(K_{2,3})) = \diam(\R(\SPK_{2,3}))=8$.
\bigskip

We show now that $\diam(\R(K_{2,6}))=20$.

Consider $T,T'$ the search trees on $\SPK_{2,6}$ shown in Figure \ref{fig:searchTreesOnSPK26}. 

\begin{figure}
\centering
\begin{tikzpicture}
\node[vertex,label=right:{\small $y_1$}] at (0,4.5) (y1) {};
\node[vertex,label=right:{\small $y_2$}] at (0,4) (y2) {};
\node[vertex,label=right:{\small $y_3$}] at (0,3.5) (y3) {};
\node[vertex,label=right:{\small $y_4$}] at (0,3) (y4) {};
\node[vertex,label=right:{\small $x_1$}] at (0,2.5) (x1) {};
\node[vertex,label=right:{\small $y_5$}] at (0,2) (y5) {};
\node[vertex,label=right:{\small $x_2$}] at (0,1.5) (x2) {};
\node[vertex,label=right:{\small $y_6$}] at (0,1) (y6) {};
\draw (y1)--(y2)--(y3)--(y4)--(x1)--(y5)--(x2)--(y6);
\node at (0,0.3) (T) {$T$};
\end{tikzpicture}
\qquad
\begin{tikzpicture}
\node[vertex,label=right:{\small $y_6$}] at (0,4.5) (y6) {};
\node[vertex,label=right:{\small $y_5$}] at (0,4) (y5) {};
\node[vertex,label=right:{\small $y_4$}] at (0,3.5) (y4) {};
\node[vertex,label=right:{\small $y_3$}] at (0,3) (y3) {};
\node[vertex,label=right:{\small $y_2$}] at (0,2.5) (y2) {};
\node[vertex,label=right:{\small $x_2$}] at (0,2) (x2) {};
\node[vertex,label=right:{\small $x_1$}] at (0,1.5) (x1) {};
\node[vertex,label=right:{\small $y_1$}] at (0,1) (y1) {};
\draw (y1)--(y2)--(y3)--(y4)--(x1)--(y5)--(x2)--(y6);
\node at (0,0.3) (T') {$T'$};
\end{tikzpicture}
    \caption{Two search trees on $\SPK_{2,6}$.}
    \label{fig:searchTreesOnSPK26}
\end{figure}

Recall that $\dist_{\R(\SPK_{2,6})}(T,T')\leq \diam(\R(\SPK_{2,6})) = 20$.

We show that $\dist_{\R(\SPK_{2,6})}(T,T')\geq 20$. 
Consider a path in $\R(\SPK_{2,6})$ from $T$ to $T'$. Let $L\subseteq Q$ the set of vertices of $Q$ that are leaves in some tree of the path. Let $k=|L|$. In this path the rotations in Table \ref{table_rotations_6} are necessary. 

\begin{table}
\centering
\begin{tabular}{ |c|c| }
\hline
No. of rotations &  Description \\
\hline
1 & exchange $x_1x_2$\\
\hline
$\alpha+ 2(k-2)$ & leaves go down, $\alpha=\begin{cases}
        1\quad \text{if }y_5\in L \\
        2\quad \text{if }y_5\notin L
        \end{cases}$ \\
\hline
$\binom{6-k}{2}$ & sort vertices of $Q$ within the handle \\
\hline
$k(6-k)$ & leaves non-leaves rotations \\
\hline
$2(k-1)$ & leaves go up \\
\hline
$\beta$ &  $\beta=\begin{cases}
        0\quad \text{if }y_5\in L \\
        1\quad \text{if }y_5\notin L  \quad (y_5x_1) 
        \end{cases}$\\
\hline
\end{tabular}
\caption{Necessary rotations in a sequence that transforms $T$ into $T'$.}
\label{table_rotations_6}
\end{table}

Then, the length of this path is at least
\[
1+\gamma+2(k-2)+\binom{6-k}{2}+k(6-k)+2(k-1),
\]
where $\gamma=\alpha+\beta=\begin{cases}
        1\quad \text{if }y_5\in L \\
        3\quad \text{if }y_5\notin L.
        \end{cases}$

It is not hard to see that this the minimum value for this expression with $2\leq k\leq 6$ and $\gamma\in\{1,3\}$ is 20 (note that if $k=2$, $y_5\notin L$). It follows that $\dist_{\R(\SPK_{2,6})}(T,T')\geq 20$, and therefore $\dist_{\R(\SPK_{2,6})}(T,T') = 20$. 

Similarly to case $q=3$, we see that if $y_6$ is a leaf again after the occurrence of a $P$-special edge (or $y_1$ is a leaf before), this requires two additional rotations (with $x_1$ and $x_2$), apart from the at least 20 rotations counted in Table \ref{table_rotations_6}.

If a $TT'$-path has a $P$-special edge and the vertices $y_1$ and $y_6$ are simultaneously leaves in some tree of the sequence, then it must be the case that $y_6$ is a leaf again after the special edge or that $y_1$ is a leaf before the edge, making this path not optimal.  If $y_1$ and $y_6$ are never leaves simultaneously, then there must be a rotation between them, since they have different relative order in $T$ and $T'$, that is, there is an additional rotation, apart from those counted in Table \ref{table_rotations_6}. This makes the path again non optimal. Hence, no $TT'$-path of minimum length in $\R(\SPK_{2,6})$ has a special edge.

It follows that $\diam(\R(K_{2,6})) = \diam(\R(\SPK_{2,6}))=20$.
\bigskip

Finally, we prove that $\diam(\R(K_{2,7}))=25$.

Consider the search trees $T,T'$ on $\SPK_{2,7}$ shown in Figure $\ref{fig:searchTreesOnSPK27}$. 

\begin{figure}
\centering
\begin{tikzpicture}
\node[vertex,label=right:{\small $y_7$}] at (0,5) (y7) {};
\node[vertex,label=right:{\small $y_6$}] at (0,4.5) (y6) {};
\node[vertex,label=right:{\small $y_5$}] at (0,4) (y5) {};
\node[vertex,label=right:{\small $y_4$}] at (0,3.5) (y4) {};
\node[vertex,label=right:{\small $y_3$}] at (0,3) (y3) {};
\node[vertex,label=right:{\small $x_2$}] at (0,2.5) (x2) {};
\node[vertex,label=right:{\small $y_2$}] at (0,2) (y2) {};
\node[vertex,label=right:{\small $x_1$}] at (0,1.5) (x1) {};
\node[vertex,label=right:{\small $y_1$}] at (0,1) (y1) {};
\draw (y1)--(x1)--(y2)--(x2)--(y3)--(y4)--(y5)--(y6)--(y7);
\node at (0,0.3) (T) {$T$};
\end{tikzpicture}
\qquad
\begin{tikzpicture}
\node[vertex,label=right:{\small $y_1$}] at (0,4.5) (y1) {};
\node[vertex,label=right:{\small $y_2$}] at (0,4) (y2) {};
\node[vertex,label=right:{\small $y_3$}] at (0,3.5) (y3) {};
\node[vertex,label=right:{\small $y_4$}] at (0,3) (y4) {};
\node[vertex,label=right:{\small $y_5$}] at (0,2.5) (y5) {};
\node[vertex,label=right:{\small $y_6$}] at (0,2) (y6) {};
\node[vertex,label=right:{\small $x_1$}] at (0,1.5) (x1) {};
\node[vertex,label=right:{\small $y_7$}] at (0,1) (y7) {};
\node[vertex,label=right:{\small $x_2$}] at (0,0.5) (x2) {};
\draw (y1)--(x1)--(y2)--(x2)--(y3)--(y4)--(y5)--(y6)--(y7);
\node at (0,-0.2) (T') {$T'$};
\end{tikzpicture}
    \caption{Two search trees on $\SPK_{2,6}$.}
    \label{fig:searchTreesOnSPK27}
\end{figure}

Recall that $\dist_{\R(\SPK_{2,7})}(T,T')\leq \diam(\R(\SPK_{2,7}))= 25$.

We prove that $\dist_{\R(\SPK_{2,7})}(T,T')\geq 25$. Consider a path from $T$ to $T'$ in $\R(\SPK_{2,7})$.

Let us consider first the case where $y_1$ is the only vertex of $Q$ that is a leaf in some tree of this sequence. Such a path has at least the following rotations. One $x_1x_2$-rotation, $\binom{7}{2}$ rotations between pairs of vertices in $Q$, two rotations of $y_1$ with both vertices in $P$, one of $y_2$ with $x_2$ (that makes $y_2$ go up) and one of $y_7$ with $x_1$ (where $y_7$ goes down). This gives a total of 26 rotations. Note that in this case, the path is not of minimum length.

Suppose now that $L\subseteq Q$ is the set of vertices in $Q$ that are leaves in some tree of this path, with $k=|L|\geq 2$. Then, in such a path the rotations in Table \ref{table_rotations_7} are necessary. 

\begin{table}
    \centering
\begin{tabular}{ |c|l| }
\hline
No. of rotations &  \multicolumn{1}{|c|}{Description}\\
\hline
1 &  exchange $x_1x_2$ \\
\hline
$\alpha+ 2(k-2)$ & leaves go down, $\alpha=\begin{cases}
        1\quad \text{if }y_2\in L \\
        2\quad \text{if }y_2\notin L
        \end{cases}$ \\
\hline
$\delta$ & sort vertices of $Q$ within handle, $\delta=\begin{cases}
\binom{7-k}{2}, \quad k\leq 5\\
0, \quad k=6,7
\end{cases}$ \\
\hline
$k(7-k)$ & leaves non-leaves rotations\\
\hline
$\beta+2(k-1)$ & leaves go up, $\beta=\begin{cases}
        1\quad \text{if }y_7\in L \\
        2\quad \text{if }y_7\notin L 
        \end{cases}$ \\
\hline
\end{tabular}
    \caption{Necessary rotations in a sequence that transforms $T$ into $T'$.}
    \label{table_rotations_7}
\end{table}

Hence, this path has length at least
\[
1+\gamma+2(k-2)+\delta+k(7-k)+2(k-1),
\]
where $\gamma=\alpha+\beta=\begin{cases}
        2\quad \text{if }y_2,y_7\in L \\
        3\quad \text{if } (y_2\in L, y_7\notin L) \lor (y_2\notin L, y_7\in L)\\
        4\quad \text{if }y_2,y_7\notin L.
        \end{cases}$
\medskip

It is not hard to see that the minimum value for this expression with $2\leq k\leq 7$ and $\gamma\in\{2,3,4\}$ is 26 (note that if $k=2$, $y_2\notin L$ or $y_7\notin L$). It follows that $\dist_{\R(\SPK_{2,7})}(T,T')\geq 25$, and therefore $\dist_{\R(\SPK_{2,7})}(T,T') = 25$.

According to the previous paragraphs, in a minimum length path $L=Q$. In particular $y_2,y_7\in L$.

Observe that in Table \ref{table_rotations_7} only 3 rotations between $y_2$ and vertices in $P$ are counted (one when $y_2$ goes down to be a leaf and 2 when it goes up). If there were a $P$-special edge in the path, and if $y_2$ were a leaf after this edge, 5 additional rotations would be needed, between $y_2$ and vertices of $P$ (1 when $x_2$ goes down, 2 when $y_2$ goes down to be a leaf, 2 it goes back up again). This makes the path of length at least 27, and therefore not optimal. Hence, $y_2$ must be a leaf exclusively before the special edge. Analogously, $y_7$ must be a leaf exclusively after the special edge. Thus, $y_2$ and $y_7$ are not leaves simultaneously in any tree of the path. This means that there must be at least one rotation between them, since they have different relative order in $T$ and $T'$, adding one rotation to those already counted in Table \ref{table_rotations_7}. Therefore, a path with a $P$-special edge is not optimal.

It follows that $\diam(\R(K_{2,7})) = \diam(\R(\SPK_{2,7}))=25$.

\section{Further remarks and future work}

The graph operations introduced in this article provide a method for generating several families of graphs. As a result, the findings presented here offer tools for investigating the structural and combinatorial properties of their corresponding rotation graphs, as well as key graph parameters such as chromatic number and diameter

From the results concerning chromatic numbers, interesting problems arise:
\begin{itemize}
    \item We observed that a rotation graph has chromatic number 2 if and only if it is a permutohedron. Since the chromatic number of rotation graphs for threshold graphs and complete bipartite graphs is 3, it would be interesting to characterize all rotation graphs with chromatic number 3. In particular, do rotation graphs of trivially perfect graphs also have chromatic number 3?
    \item It is well-established that trivially perfect graphs can be constructed from a single-vertex graph through the operations of disjoint union of trivially perfect graphs and addition of a vertex whose open neighborhood forms a trivially perfect graph. Therefore, a compelling approach to addressing the previous question is to determine whether adding a universal vertex to a graph preserves the chromatic number of its corresponding rotation graph.
    \item We have established that under certain conditions, adding a vertex to a graph preserves the chromatic number of the corresponding rotation graph. For example, adding a pendant vertex to a universal vertex $v$ preserves the chromatic number, though this might not be the case if $v$ is not universal. As noted in \cite{ARSW-2018}, $\chi(\R(P_{10}))\geq 4$, where $P_{10}$ is the path on 10 vertices. A natural next step is to identify which operations on a graph lead to an increase in the chromatic number of its corresponding rotation graph.
\end{itemize}

Further questions on distances and diameter of rotation graphs emerge:
\begin{itemize}
    \item An appealing direction for future work would be to derive bounds for distances, similar to those established in Theorem \ref{theo_lower_bound_quotient}, where analyzing the structure of modified rotation graphs enables bounding the distances between pairs of vertices. Ultimately, determining precise distances in rotation graphs generated through the described graph operations could provide deeper insights into their structural properties, and the computation of their diameters.
    
    \item Let $G$ be a graph with a subset $W$ of true twins and $S=E(G[W])$. 
    By Theorem \ref{theo_lower_bound_quotient} we have that $0\leq\diam(\R(G)) - \diam(\R(G-S)) \leq \binom{|W|}{2}$. As we saw for the case of $K_{2,q}$ and $\SPK_{2,q}$, when $|W|=2$ both possible values for $\diam(\R(G)) - \diam(\R(G-S))$ are achieved. Is this the case for every size of $W$? In other words, given $j$ and $0\leq k\leq \binom{j}{2}$, is there a graph $G$ with a subset $W$ of true twins such that $|W|=j$ and that $\diam(\R(G)) - \diam(\R(G-S))=k$?
\end{itemize}


\printbibliography

\end{document}